\documentclass[12pt]{amsart}
\usepackage[T1]{fontenc}
\usepackage[latin1]{inputenc}
\usepackage{typearea}
\usepackage{geometry}
\usepackage{ulem}
\usepackage{amsmath}
\usepackage{amssymb}
\usepackage{latexsym}
\usepackage{enumerate}
\usepackage{amsthm}
\usepackage[all]{xy}
\usepackage{hhline}
\usepackage{epsf} 
\usepackage{cite}

\newtheorem{theorem}{{\sc Theorem}}[section]
\newtheorem{cor}[theorem]{{\sc Corollary}}

\newtheorem{lemma}[theorem]{{\sc Lemma}}
\newtheorem{prop}[theorem]{{\sc Proposition}}

\theoremstyle{remark}
\newtheorem{remark}[theorem]{{\sc Remark}}

\theoremstyle{definition}

\newtheorem{cond}[theorem]{\sc condition}

\newcommand{\R}{\mathbb{R} }
\newcommand{\N}{\mathbb{N} }

\newcommand{\A}{\mathcal{A}}
\newcommand{\B}{\mathcal{B}}
\newcommand{\F}{\mathcal{F}}

\newcommand{\D}{\mathcal{D}}
\newcommand{\K}{\mathcal{K}}
\newcommand{\calH}{\mathcal{H}}
\newcommand{\calL}{\mathcal{L}}

\newcommand{\la}{\lambda}
\newcommand{\al}{\alpha}
\newcommand{\be}{\beta}
\newcommand{\ga}{\gamma}

\newcommand{\Om}{\Omega}

\newcommand{\albe}{_{\al,\be}}
\newcommand{\abquer}{\overline{(a,b)}}
\newcommand{\htilde}{\tilde{h}}

\providecommand{\abs}[1]{\lvert #1\rvert}

\providecommand{\fnorm}[1]{\lVert #1\rVert_\infty}

\DeclareMathOperator{\supp}{supp}

\renewcommand{\phi}{\varphi}
\renewcommand{\epsilon}{\varepsilon}

\renewcommand{\rho}{\varrho}

\begin{document}
\title[Stein's method for absolutely continuous distributions]{Stein's method of exchangeable pairs for absolutely continuous, univariate distributions with applications to the Polya urn model}
\author{Christian D\"obler}
\thanks{The author has been supported by Deutsche Forschungsgemeinschaft via SFB-TR 12.\\
Ruhr-Universit\"at Bochum, Fakult\"at f\"ur Mathematik, NA 3/68, D-44780 Bochum, Germany. \\
christian.doebler@ruhr-uni-bochum.de\\
{\it Keywords:} Stein's method, Beta distributions, exchangeable pairs, Polya urn, density approach}
\begin{abstract}
We propose a way of finding a Stein type characterization of a given absolutely continuous distribution $\mu$ on $\R$ which is motivated by a regression property satisfied by an exchangeable pair $(W,W')$ where $\calL(W)$ is supposed or known to be close to $\mu$. We also develop the exchangeable pairs approach within this setting. This general procedure is then specialized to the class of Beta distributions and as an application, a convergence rate for the relative number of drawn red balls among the first $n$ drawings from a Polya urn is computed. 
\end{abstract}

\maketitle

\section{Introduction}\label{intro}
Since its introduction in \cite{St72} in 1972 Stein's method has become a very famous and useful tool for proving distributional convergence. One of its main advantages is, that it usually automatically yields a concrete error bound for various distributional distances. Being first only developed for normal approximation, it was observed by several authors, that Stein's technique of characterizing operators is by no means restricted to the normal distribution (see, e.g. \cite{Ch75} for the Poisson distribution, \cite{St86} and \cite{DHRS} for absolutely continuous distributions on $\R$). Highly recommendable accounts of Stein's method are the books \cite{CGS} for the normal approximation and \cite{BarCh} for a general introduction into the method.\\
In this paper we consider an absolutely continuous distribution $\mu$ on $(\R,\B)$ with density function $p$ and corresponding distribution function $F$. Let $-\infty\leq a<b\leq\infty$ be extended real numbers, such that $\supp\bigl(\mu\bigr)\subseteq\abquer$ and such that for each choice of real numbers $a'>a$ and $b'<b$ we have that $\supp\bigl(\mu\bigr)\not\subseteq\overline{(a',b)}$ and 
$\supp\bigl(\mu\bigr)\not\subseteq\overline{(a,b')}$, i.e. if $a$ is real, then it is the left endpoint of the support of $\mu$ and if $b$ is real it is the right endpoint of $\supp\bigl(\mu\bigr)$. Here and in what follows, the closing operation is with respect to the usual topology on $\R$. If the density $p$ is positive on $(a,b)$ and absolutely continuous on every compact subinterval of $(a,b)$, then the so called \textit{density approach} within Stein's method gives a Stein identity and even a Stein characterization for $\mu$ (see \cite{DHRS}, \cite{ChSh} or\cite{CGS}). In short, a real valued random variable $X$ has distribution $\mu$ if and only if for each function $f$ in a suitable class $\F$ we have

\begin{equation}\label{steiniddens}
E\bigl[f'(X)+\psi(X)f(X)\bigr]=0\,, 
\end{equation}

where $\psi(x):=\frac{d}{dx}\log p(x)$ is the logarithmic derivative of $p$. For a given Borel-measurable test function $h$ with 
$\int_\R\abs{h(x)}d\mu(x)<\infty$ this characterization motivates the so called \textit{Stein's equation} 

\begin{equation}\label{steineqdens}
f'(x)+\psi(x)f(x)=h(x)-\mu(h)\,, 
\end{equation}
to be solved for $f$, where we write $\mu(h)$ for $\int_\R hd\mu$. It turns out, that a solution $f_h$ to (\ref{steineqdens}) on $(a,b)$ is given by 

\begin{equation}\label{steinsoldens}
f_h(x):=\frac{1}{p(x)}\int_{-\infty}^x\bigl(h(y)-\mu(h)\bigr)p(y)dy=-\frac{1}{p(x)}\int_x^\infty \bigl(h(y)-\mu(h)\bigr)p(y)dy
\end{equation}

and that, if $f_h$ is bounded and $\frac{1}{p}$ is unbounded on $(a,b)$, then $f_h$ is the only bounded solution on $(a,b)$.
For general properties of the solutions $f_h$ see \cite{ChSh} or \cite{CGS}. Note that for general Borel-measurable $h$ it cannot be expected that there exists a solution $f$ which is differentiable on all of $(a,b)$ and satisfies (\ref{steineqdens}) pointwise. Thus, a solution is understood to be an almost everywhere differentiable and Borel-measurable function which satisfies (\ref{steineqdens}) at all points $x\in(a,b)$ where it is in fact differentiable and contrary to the usual convention, at the remaining points $x\in(a,b)$ one defines 
$f'(x):=-\psi(x)+h(x)-\mu(h)$. This yields a Borel-measurable function $f'$ on $(a,b)$ such that (\ref{steineqdens})
holds for each $x\in(a,b)$.\\
In order to understand the exchangeable pairs technique in the framework of the density approach it might be helpful to recall the exchangeable pairs method in the situation of normal approximation. This method, which was first presented in Stein's monograph \cite{St86}, is a cornerstone of Stein's method of normal approximation and is still the most frequently used coupling. This is due to the wide applicability of standard couplings like the \textit{Gibbs sampler} or making one time step in a reversible Markov chain, which generally yield exchangeable pairs. By definition, an \textit{exchangeable pair} is a pair $(W,W')$ of random variables, defined on a common probability space, such that their joint distribution is symmetric, i.e. such that $(W,W')\stackrel{\D}{=}(W',W)$. In \cite{St86}, in order to show that a given real-valued random variable $W$ is approximately standard normally distributed, Stein proposes the construction of another random variable $W'$, a small random perturbation of $W$, on the same space as $W$ such that $(W,W')$ forms an exchangeable pair and additionally the following \textit{linear regression property} holds:

\begin{equation} \label{linreg}
E[W'-W|W]=-\la W   
\end{equation}

Here, $\la\in(0,1)$ is a constant which is typically close to zero for conveniently chosen $W'$. If this condition is satisfied, then the distributional distance of $\calL(W)$ to $N(0,1)$ can be efficiently bounded in various metrics, including the Kolmogorov and Wasserstein metrics (see, e.g. \cite{St86}, \cite{CheShaSing} or \cite{CGS} for the common ``plug-in theorems'').\\
 The range of examples to which this technique could be applied was considerably extended by the work \cite{RiRo97} of Rinott and Rotar who proved normal approximation theorems allowing the linear regression property to be satisfied
only approximately. Specifically, they assumed the existence of a random quantity $R$, which is dominated by $\lambda W$ in size, such that 
\begin{equation}\label{linregr}
E[W'-W|W]=-\lambda W +R\,. 
\end{equation}

Note, that necessarily $R$ is $\sigma(W)$-measurable and that unlike condition (\ref{linreg}), condition (\ref{linregr}) is not a true condition on the pair $(W,W')$ since we can always define $R:=E[W'-W|W]+\lambda W$ for each given constant $\la>0$. However, the ``plug-in theorems'' in \cite{RiRo97}, \cite{ShSu} or \cite{CGS} clarify that $R$ has to be of smaller order than $\la$ in order to yield useful bounds. Since $W$ is supposed to have a ``true'' distributional limit, it follows that both, $\la$ and $R$, are at least asymptotically unique (see also the introduction of \cite{ReiRoe09} for the discussion of this topic).\\
When dealing with our possibly non-normal distribution $\mu$, the question is what condition to substitute for the linear regression property (\ref{linreg}) or (\ref{linregr}). This question was succesfully answered independently by Eichelsbacher and L\"{o}we in \cite{EiLo10} and by Chatterjee and Shao in \cite{ChSh}. They pointed out, that in this more general setting the appropriate regression property is  

\begin{equation}\label{regdens}
E[W'-W|W]=\la \psi(W) +R\,,
\end{equation}

where, again $\la>0$ is constant and $R$ is of smaller order than $\la \psi(W)$.
In order to give a flavour of the resulting ``plug-in theorems'', we present parts of Theorem 2.4 from \cite{EiLo10}.

\begin{theorem}\label{theoeilo}
Let the density $p$ be positive and absolutely continuous on $(a,b)$. Suppose that there exist positive real constants $c_1,c_2$ and $c_3$ such that for any Lipschitz-continuous function $h:\abquer\rightarrow\R$ with minimal Lipschitz constant $\fnorm{h'}$, the solution $f_h$ given by (\ref{steinsoldens}) satisfies

\[\fnorm{f_h}\leq c_1\fnorm{h}\,,\quad \fnorm{f_h'}\leq c_2\fnorm{h'}\quad\text{and}\quad\fnorm{f_h''}\leq c_3\fnorm{h'}\,,\]

where, for a function $f:\abquer\rightarrow\R$ we let $\fnorm{f}$ be its essential supremum norm.\\
Then, for any exchangeable pair $(W,W')$ satisfying condition (\ref{regdens}) we have 

\begin{eqnarray}\label{eilotheo}
&&|E[h(W)]-\mu(h)|\\
&\leq& \fnorm{h'}\Biggl(c_2E\biggl[\Bigl|1-\frac{1}{2\la}E\bigl[(W-W')^2|W\bigr]\Bigr|\biggr]+\frac{c_3}{4\la}E\bigl[|W'-W|^3\bigr]
+\frac{c_1}{\la}\sqrt{E[R^2]}\Biggr)\,.\nonumber 
\end{eqnarray}
\end{theorem}

The third term of the bound (\ref{eilotheo}) reveals that, in fact, $R$ must be of smaller order than $\la$ in order for the bound to be useful and from the second term we conclude that $\la$ should such that $E\left[(W'-W)^2\right]\approx 2\la$.
The first term appearing in the bound on the right hand side of (\ref{eilotheo}) is usually interpreted such, that the random variable $ E\bigl[(W-W')^2|W\bigr]/2\la$ must ``obey a law of large numbers'' to obtain decreasing bounds.
Bounding this term is often decisive for the success of applying Theorem \ref{theoeilo}. \\

Having discussed the method of exchangeable pairs within the density approach, we now address the problem, that condition (\ref{regdens}) with negligible remainder term $R$ is in some examples not satisfied by an exchangeable pair, which, however, appears natural to us for our approximation problem. For example, in many situations where the exchangeable pair $(W,W')$ is constructed via the Gibbs sampler, we have a regression property of the form 

\begin{equation}\label{regmean}
E[W'-W|W]=\la\Bigl(-c\bigl(W-E[Z]\bigr)\Bigr)+R\,,
\end{equation}

where $\la, c>0$ are constants (the reason why $\la$ and $c$ are not subsumed into a single constant will become clear later on) and where, again, $R$ is a negligible remainder. Here, again $Z\sim\mu$. Following the theory of the paper \cite{ChSh}, condition (\ref{regmean}) suggests approximating $W$ with a normal distribution with mean $E[Z]$ and variance $\frac{1}{c}$. 
But there are situations, where the exchangeable pair $(W,W')$ is good, meaning that the difference $|W'-W|$ is ``small'', condition (\ref{regmean}) is satisfied and where we \textit{know} that $W$ is approximately distributed as a non-normal random variable $Z\sim\mu$ and so the normal approximation is inappropriate. In general, this either means that the above discussed law of large numbers cannot hold or that the resulting error term $R$ in (\ref{regdens}) is not negligible. 
These observations motivate a new version of Stein's method, that allows for a more general regression property.

Suppose, that an appropriately chosen exchangeable pair $(W,W')$ satisfies the following \textit{general regression property}:

\begin{equation}\label{genreg}
E[W'-W|W]=\la\ga(W)+R\,, 
\end{equation}

where $\la>0$ is constant, $\ga$ is a measurable function, whose domain contains $\abquer$ and will be discussed further below and where $R$ is a negligible remainder term. We will see, that it will be advantageous if the term $\ga(x)\cdot g(x)$ appears in the ``new'' Stein equation. So we make the following ansatz for the Stein identity:

\begin{equation}\label{gensteinid}
E[\eta(Z)g'(Z)+\ga(Z)g(Z)]=0\,,
\end{equation}

where $\eta$ is another function which still has to be found.

Starting from the Stein identity (\ref{gensteinid}) our aim is to identify the function $\eta$. If this approach is succesful, the  Stein equation corrersponding to a meaurable function $h$ will be

\begin{equation}\label{gensteineq}
\eta(x)g'(x)+\ga(x)g(x)=h(x)-\mu(h)\,.
\end{equation}

Let $f_h$ be the solution (\ref{steinsoldens}) to the equation (\ref{steineqdens}). For the solution $g_h$ of (\ref{gensteineq}) we make the ansatz $g_h(x)=\al(x)f_h(x)$ for some (sufficiently smooth) function $\al$. We will always let $\htilde:=h-\mu(h)$ and obtain

\begin{eqnarray}
&&\eta(x)g_h'(x)=\eta(x)\bigl(\al'(x)f_h(x)+\al(x)f_h'(x)\bigr)\nonumber\\ 
&=&\eta(x)\al'(x)f_h(x)+\eta(x)\al(x)\bigl(\htilde(x)-\eta(x)\psi(x)f_h(x)\bigr)\nonumber\\
&=&\eta(x)\al'(x)f_h(x)+\eta(x)\al(x)\htilde(x)-\eta(x)\psi(x)g_h(x)\label{gensteincal1}\\
&\stackrel{!}{=}&\htilde(x)-\ga(x)g_h(x)\nonumber\,.
\end{eqnarray}

For this identity to hold, irrespective of the test function $h$, it must be the case that $\al(x)=\frac{1}{\eta(x)}$ (particularly $\eta$ must be differentiable at least almost everywhere) and hence

\begin{equation}\label{ablal}
\al'(x)=\frac{-\eta'(x)}{\eta(x)^2}= -\frac{\eta'(x)}{\eta(x)}\al(x)\,.
\end{equation}

Plugging this into (\ref{gensteincal1}) we obtain 

\[\eta(x)g_h'(x)=-\eta'(x)g_h(x)-\eta(x)\psi(x)g_h(x)+\htilde(x)\,.\]

This equals $\htilde(x)-\ga(x)g_h(x)$ if and only if $\eta$ satisfies the ordinary differential equation 

\begin{equation}\label{etadgl}
\eta'(x)=\ga(x)-\psi(x)\eta(x)\,.
\end{equation}

This is a first order linear differential equation, which can of course be solved explicitly by the method of variation of the constant. It turns out, that the \textit{right} solution is given by 
\begin{equation}\label{etaformel}
\eta(x)=\frac{1}{p(x)}\int_a^x\ga(t)p(t)dt\,,
\end{equation}

at least, if $\int_a^b|\ga(t)|p(t)dt=E[|\ga(Z)|]<\infty$. But this is a very natural condition to hold, since the function $\ga$ was motivated by the regression property (\ref{genreg}) and so, if the random variables $W$ and $W'$ are integrable we obtain that both sides of (\ref{genreg}) must be $P$-integrable and, in fact

\[E[\la\ga(W)+R]=E\bigl[E[W'-W|W]\bigr]=E[W']-E[W]=0\,.\]

Neglecting the remainder term $R$ we thus see that $E[\ga(W)]$ should exist and, in fact, be close to zero. So since $W\stackrel{\D}{\approx}Z$ we find it reasonable that $E[\ga(Z)]$ exists and even equals zero. Furthermore, it is a matter of routine to check, that $\eta$ as given in (\ref{etaformel}) indeed still satisfies (\ref{etadgl}).    
The above calculations starting with (\ref{gensteinid}) were rather formal but crucial for the motivation and understanding of our approach. The paper is organized in the following way. Rigorous results and the abstract theory for general $\mu$ are presented in Section \ref{mainresultsgen}. These results are then further spezialized to the Beta distributions in Section \ref{mainresultsbeta}. In Section \ref{applications} the theory combined with a suitable exchangeable pairs coupling is used to prove a rate of convergence of order $\frac{1}{n}$ in a Polya urn model (see Theorem \ref{polyarate}). In Section \ref{appendix} some rather lengthy or technical proofs can be found and a sufficiently general version of de l' H\^{o}pital's rule for merely absolutely continuous functions is provided. This result justifies all the invocations of this famous rule in the present work.

\section*{Acknowledgements}
A few days after this work was on the arXiv, G. Reinert and L. Goldstein posted a preprint (see \cite{GolRei12}), which also develops Stein's method for the Beta distributions and uses a comparison technique to prove error bounds of order $n^{-1}$ in the Wasserstein distance for the more special Polya urn model, where the drawn ball is replaced to the urn together with only one extra ball of the same colour.

\section{The general theory}\label{mainresultsgen}
In this section we state the main results concerning bounds on the solution of the Stein equation (\ref{gensteineq}) and its first two derivatives and present a general ``plug-in theorem'' within a general exchangeable pairs approach.\\
In order to derive precise results, we have to impose conditions on the density $p$ of $\mu$ and on the coefficient $\ga$ from 
(\ref{gensteineq}), which was motivated by (\ref{genreg}). In applications it might be the case that a sequence of random variables of unbounded support converges in distribution to a random variable $Z$ with bounded support. Usually Stein's method for the distribution $\mu$ of $Z$ cannot handle such approximation problems, since the Stein equation and its solution are only defined on $(a,b)$ or $\abquer$. However, in many practical cases, the coefficients in Stein's equation are given by certain ``analytical expressions'', which also make sense on $\R\setminus(a,b)$ and hence, also the Stein solution can be defined there. In order to cover these situations, too, we will have to assume a number of conditions in this section. Those readers, who are only interested in approximation problems for random sequences, whose support is contained in $\abquer$ may restrict their attention to the relevant results. For reasons of readability most proofs are deferred to the appendix, only some short ones and the proof of Proposition \ref{genpluginprop1} are given in this section. \\
In the following we always let $Z$ be a real-valued random variable with distribution $\mu$. We will always assume the following conditions on the density $p$ and the coefficient $\ga$:

\begin{cond}\label{gencond1}
The density $p$ is positive on the interval $(a,b)$ and absolutely continuous on every compact interval $[c,d]\subseteq(a,b)$.
\end{cond}

\begin{cond}\label{gencond2}
The function $\ga:\abquer\rightarrow\R$ is such that 
\begin{enumerate}[(i)]
\item $\ga$ is continuous on $\abquer$
\item $\ga$ is strictly decreasing on $\abquer$
\item $\int_a^b |\ga(t)|p(t)dt<\infty$ and in fact $E[\ga(Z)]=\int_a^b\ga(t)p(t)dt=0$
\item There is a unique $x_0\in\abquer$ with $\ga(x_0)=0$.
\end{enumerate} 
\end{cond}

\begin{remark}\label{genrem1}
Note that in Condition \ref{gencond2} (iv) is actually implied by (i),(ii) and (iii) and the intermediate value theorem. Furthermore, (iv) implies that $\ga$ is positive on $(a,x_0)$ and is negative on $(x_0,b)$. 
\end{remark}

By item (iii) in Condition \ref{gencond2} we can define the function $I:\abquer\rightarrow\R$ by 
$I(x):=\int_a^x\ga(t)p(t)dt$ which is continuously differentiable on $\abquer$ if $\ga$ satisfies Condition \ref{gencond2}.

\begin{prop}\label{genprop1}
Under Conditions \ref{gencond1} and \ref{gencond2} the function $I$ has the following properties:
\begin{enumerate}[{\normalfont (a)}] 
\item $I(x)=-\int_x^b\ga(t)p(t)dt$
\item $I(x)>0$ for each $x\in(a,b)$
\item $I$ is strictly increasing on $(a,x_0)$ and strictly decreasing on $(x_0,b)$ and hence attains its global maximum at $x_0$.
\end{enumerate}
\end{prop}

\begin{proof}
Of course, (a) is immediately implied by item (iii) from Condition \ref{gencond2}.\\
To prove (b) and (c) first observe that by (iii) we have $\lim_{x\searrow a} I(x)=0=\lim_{x\nearrow b} I(x)$. Furthermore,  
$I'(x)=\ga(x)p(x)$ is postive on $(a,x_0)$ and negative on $(x_0,b)$ implying the results.
\end{proof}

We begin developing Stein's method for the distribution $\mu$ satisfying Condition \ref{gencond1} and the coefficient $\ga$ of the unknown function $g$ in (\ref{gensteineq}), which we assume to fulfill Condition \ref{gencond2}. We \textit{define} the function $\eta:(a,b)\rightarrow\R$ by (\ref{etaformel}) and for a given Borel-measurable test function $h$ with 
$E[|h(Z)|]<\infty$ we consider the \textit{Stein equation} (\ref{gensteineq}). In most cases of practical interest we will have $\lim_{x\searrow a}\eta(x)=0$ if $a>-\infty$ so that we can define $\eta(a)=0$ to obtain a function which is continuous at $a$. The same remark holds for $b$ (see also Proposition \ref{genprop2} down below and the discussion following it). Note that we can also write $\eta(x)=\frac{I(x)}{p(x)}$ and so we can infer several properties of $\eta$ from those of the function $I$.

\begin{prop}\label{genprop2}
Under Conditions \ref{gencond1} and \ref{gencond2} the function $\eta$ has the following properties:
\begin{enumerate}[{\normalfont (a)}] 
\item $\eta$ is positive on $(a,b)$, absolutely continuous on every compact subinterval $[c,d]\subseteq(a,b)$ and 
$\eta'(x)=\ga(x)-\psi(x)\eta(x)$ for $\la$-almost all $x\in(a,b)$.
\item $\lim_{x\searrow a}\eta(x) p(x)=\lim_{x\nearrow b}\eta(x)p(x)=0$
\item If $\lim_{x\searrow a}p(x)=0$, then 
$\lim_{x\searrow a}\eta(x)=\frac{\ga(a)}{\lim_{x\searrow a}\psi(x)}$, if this limit exists. 
\item If $\liminf_{x\searrow a}p(x)\in(0,\infty)\cup\{\infty\}$ then $\lim_{x\searrow a}\eta(x)=0$
\item If $\lim_{x\nearrow b}p(x)=0$, then 
$\lim_{x\nearrow b}\eta(x)=\frac{\ga(b)}{\lim_{x\nearrow b}\psi(x)}$, if this limit exists.
\item If $\liminf_{x\nearrow b}p(x)\in(0,\infty)\cup\{\infty\}$ then $\lim_{x\nearrow b}\eta(x)=0$

\end{enumerate}
\end{prop} 

\begin{proof}
The first part of (a) follows from the fact, that $I$ is positive on $(a,b)$ and $C^1$ on $\abquer$ and hence absolutely continuous on $[c,d]$ and that $p$ is also absolutely continuous and  bounded below by a positive constant on $[c,d]$. The rest of (a) has already been observed.  Items (b), (d) and (f) follow immediately from the properties of the function $I$ in Proposition \ref{genprop1}. To prove (c), we use de l'H\^{o}pital's rule (see Theorem \ref{Hospital}) to derive

\[\lim_{x \searrow a}\eta(x)=\lim_{x \searrow a}\frac{I(x)}{p(x)}=\lim_{x \searrow a}\frac{\ga(x)p(x)}{p'(x)}
=\frac{\ga(a)}{\lim_{x \searrow a}\psi(x)}\,.\]

In a similar way one can prove (e).
\end{proof}

The following ``Mill's ratio'' condition on the density $p$ and the corresponding distribution function $F$ is often satisfied and will yield $\lim_{x\searrow a}\eta(x)=\lim_{x\nearrow b}\eta(x)=0$.

\begin{cond}\label{condboundsupp}
The density $p$ of $\mu$ satisfies all the properties from Condition \ref{gencond1} and also the following:
\begin{enumerate}[(i)]
\item If $a>-\infty$, then $\lim_{x\searrow a}\frac{F(x)}{p(x)}=0$.
\item If $b<\infty$, then $\lim_{x\nearrow b}\frac{1-F(x)}{p(x)}=0$.
\end{enumerate}
\end{cond}

\begin{remark}\label{remcondboundsupp}
\begin{enumerate}[(a)]
\item Condition \ref{condboundsupp} is always satisfied if the density $p$ is bounded away from zero in suitable neighbourhoods of $a$ and $b$.
\item Assume that both, $a>-\infty$ and $b<\infty$ and that $\lim_{x\searrow a}p(x)=\lim_{x\nearrow b}p(x)=0$. Then Condition \ref{condboundsupp} is satisfied, if there is a $\delta>0$ such that $p$ is increasing on $(a,a+\delta)$ and 
decreasing on $(b-\delta,b)$. This is easily seen by the inequality

\[F(x)=\int_a^x p(t)dt\leq p(x)(x-a)\,,\]

valid for $x\in(a,a+\delta)$ and a similar one for the right end point $b$.

\item Suppose that $a>-\infty$ and $b<\infty$, that $\lim_{x\searrow a}p(x)=\lim_{x\nearrow b}p(x)=0$ and that there is a $\delta>0$ such that $p$ is convex on $(a,a+\delta)$ and on $(b-\delta,b)$. Then the assumptions of (b) and hence Condition  
\ref{condboundsupp} is satisfied. 
In fact, first we can extend $p$ to a continuous and convex function on $[a,b)$ 
by setting $p(a):=0$. Now, let $a<x<y<a+\delta$. Then, there exists a $\la\in(0,1)$ with $x=\la a+(1-\la)y$ and by convexity we have:

\begin{eqnarray*}
p(y)-p(x)&=&p(y)-p\bigl(\la a+(1-\la)y\bigr)\geq p(y)-\la p(a)-(1-\la)p(y)\\
&=&\la p(y)>0
\end{eqnarray*}

Thus, $p$ is strictly increasing on $(a,a+\delta)$. Similarly, one shows, that $p$ is strictly decreasing on $(b-\delta,b)$, if $p$ is convex there. 

\item If $p$ is analytic at $a$ and $b$, then Condition \ref{condboundsupp} is also satisfied. Indeed, if there is an 
$r>0$ such that $p(x)=\sum_{k=0}^\infty c_k (x-a)^k$ for all $x\in(a,a+r)$, then the function $f:(a-r,a+r)\rightarrow\R$ with $f(x):=\sum_{k=0}^\infty c_k (x-a)^k$ is well-defined. Let $n_0:=\min\{k\geq0\:c_k\not=0\}$. Then $n_0<\infty$ since 
$\supp\bigl(\mu\bigr)=\abquer$. If $n_0=0$ and hence $f(a)=c_0=\lim_{x\searrow a}p(x)\not=0$, then there is nothing to show. Otherwise, we have $p(x)=(x-a)^{n_0}\sum_{k=n_0}^\infty c_k(x-a)^{k-n_0}$ and 
$p'(x)=(x-a)^{n_0-1}\sum_{k=n_0}^\infty kc_k (x-a)^{k-n_0}$ for $x\in(a,a+r)$ and hence, by de l'H\^{o}pital's rule

\begin{eqnarray*}
\lim_{x\searrow a}\frac{F(x)}{p(x)}&=&\lim_{x\searrow a}\frac{p(x)}{p'(x)}
=\lim_{x\searrow a}(x-a)\frac{\sum_{k=n_0}^\infty c_k(x-a)^{k-n_0}}{\sum_{k=n_0}^\infty kc_k (x-a)^{k-n_0}}\\
&=&\frac{c_{n_0}}{n_0c_{n_0}}\lim_{x\searrow a}(x-a)=0\,.
\end{eqnarray*}
\qed
 
\end{enumerate}
\end{remark}

The following proposition provides the announced result.

\begin{prop}\label{etaprop}
Assume Condition \ref{condboundsupp}. Then the function $\eta$ vanishes at the finite end points of the support $\abquer$ of $\mu$, i.e. if $a>-\infty$, then $\lim_{x\searrow a}\eta(x)=0$ and if $b<\infty$, then $\lim_{x\nearrow b}\eta(x)=0$.
Hence, we may extend $\eta$ to a continuous function on $\abquer$ by letting $\eta(a):=\eta(b):=0$.
\end{prop}

\begin{proof}
Suppose, that $a>-\infty$. Then, by the positivity of $I$ and the monotonicity of $\ga$, for $a<x<b$:

\[0<I(x)=\int_a^x\ga(t)p(t)dt<\ga(a)\int_a^x p(t)dt=\ga(a)F(x)\]

Hence, 

\begin{eqnarray*}
0\leq\liminf_{x\searrow a}\eta(x)\leq\limsup_{x\searrow a}\eta(x)=\limsup_{x\searrow a}\frac{I(x)}{p(x)}
\leq\ga(a)\lim_{x\searrow a}\frac{F(x)}{p(x)}=0\,,
\end{eqnarray*}

so that $\lim_{x\searrow a}\eta(x)=0$.
The proof of $\lim_{x\nearrow b}\eta(x)=0$ for finite $b$ is similar by using the representation $I(x)=-\int_x^b \ga(t)p(t)dt$ and is therefore omitted.  
\end{proof}

From our above heuristic calculations we know that the function $g_h:(a,b)\rightarrow\R$ with 

\begin{eqnarray}\label{gensteinsol}
g_h(x)&:=&\frac{1}{p(x)\eta(x)}\int_a^x(h(t)-\mu(h))p(t)dt\nonumber\\
&=&-\frac{1}{p(x)\eta(x)}\int_x^b(h(t)-\mu(h))p(t)dt
\end{eqnarray}

solves the Stein equation (\ref{gensteineq}) for $x\in(a,b)$. This can also be proved directly by differentiation and the formula for $g_h$ could also be derived by the method of variation of the constant using the fact that $\log(p\cdot\eta)$ is a primitive function of $\frac{\ga}{\eta}$, which follows from (\ref{etadgl}). If we can show that $g_h$ is bounded, then it will immediately follow from Proposition \ref{genprop2} (a) that $g_h$ is the only bounded solution of 
(\ref{gensteineq}), since the solutions of the corresponding homogeneous equation are constant multiples of $\frac{1}{p\cdot\eta}$.\\
Since we do not exclude approximating random variables which take on the values $a$ or $b$, we show that the solution $g_h$ can be extended continuously to $a$ and $b$, if $h$ is continuous there.
By the properties of the function $I=p\eta $ on $(a,b)$,  from Proposition \ref{genprop2}, the continuity of $\ga$ and by de l'H\^{o}pital's rule (see Theorem \ref{Hospital}) we have 

\begin{equation}\label{conta1}
\lim_{x\searrow a}g_h(x)= \lim_{x\searrow a}\frac{\int_a^x\htilde(t)p(t)dt}{I(x)}
=\lim_{x\searrow a}\frac{\htilde(x)p(x)}{\ga(x)p(x)}=\lim_{x\searrow a}\frac{\htilde(x)}{\ga(x)}
=\frac{\htilde(a)}{\ga(a)}
\end{equation}

and, similarly, $\lim_{x\nearrow b}g_h(x)=\frac{\htilde(b)}{\ga(b)}$, if $h$ is continuous at $a$ and $b$.
Thus, we can conclude the following proposition.

\begin{prop}\label{propcont}
Assume Conditions \ref{gencond1} and \ref{gencond2} and let $h:\R\rightarrow\R$ be a Borel-measurable test function with 
$E\bigl[\abs{h(Z)}\bigr]<\infty$ and being continuous at $a$ and $b$. Then the Stein solution $g_h$ as defined above may be extended to a continuous function on $\R$ by letting 

\[g_h(a):=\frac{h(a)-\mu(h)}{\ga(a)}\quad\text{and}\quad g_h(b):=\frac{h(b)-\mu(h)}{\ga(b)}\,.\]

\end{prop}

As is typical for Stein's method, its success within the applications considerably depends on good bounds on the solutions $g_h$ and their derivative(s), generally uniformly over some given class $\calH$ of test functions $h$.\\
The next step will be to prove such bounds. It has to be mentioned that we cannot expect to derive concrete good bounds in full generality, but that sometimes further conditions have to be imposed either on the distribution $\mu$ (e.g. through the density $p$) or on the coefficient $\ga$. Nevertheless, we will derive bounds involving functional expressions which can a posteriori be simplified, computed or further bounded for concrete distributions. So our abstract viewpoint will pay off. Moreover, some of our bounds will actually hold in complete generality.\\

The next Proposition contains a bound for the solutions $g_h$ for bounded and Borel-measurable test functions $h$.

\begin{prop}\label{genprop3}
Assume Conditions \ref{gencond1} and \ref{gencond2} and let $m$ be a median for $\mu$. Then, for $h:\abquer\rightarrow\R$ Borel-measurable and bounded we have

\begin{equation}\label{genboundbounded}
 \fnorm{g_h}\leq\frac{\fnorm{h-\mu(h)}}{2I(m)}=\frac{\fnorm{h-\mu(h)}}{2\int_a^{m}\ga(t)p(t)dt}
\end{equation}
 
\end{prop}

The proof is deferred to the appendix.
The following corollary specializes this result to the case that $\ga(x)=-c(x-E[Z])$ and that $\mu$ is symmetric with respect to its median, which is then equal to its expected value $E[Z]$, that is $Z-m\stackrel{\D}{=}m-Z$.

\begin{cor}\label{gencor1}
In addition to Conditions \ref{gencond1} and \ref{gencond2} assume that the distribution $\mu$ is symmetric with respect to $m=E[Z]$ and that 
$\ga(x)=-c(x-E[Z])$ for some positive constant $c$. Then for each bounded and  Borel-measurable test function $h:\abquer\rightarrow\R$ we have

\begin{equation}\label{boundboundedsym}
\fnorm{g_h}\leq\frac{\fnorm{h-\mu(h)}}{cE[\abs{Z-m}]}
\end{equation} 
\end{cor}
   
\begin{proof}
In this case we clearly have $I(m)=\frac{c}{2} E[\abs{Z-m}]$ which implies the result by Proposition \ref{genprop3}. 
\end{proof}

In the case that $\mu=N(0,1)$ and $c=1$ this result specializes to the well known bound $\sqrt{\frac{\pi}{2}}\fnorm{h-\mu(h)}$ (see \cite{CGS} or \cite{CheShaSing}, e.g.).

\begin{remark}\label{genrem2}
In the formulation of Proposition \ref{genprop3} it might suprise that there is no bound mentioned for $\fnorm{g_h'}$. This is because, in general a bound of the form $\fnorm{g_h'}\leq C \fnorm{h}$ does not exist with a finite constant $C$ in this setup. Note that this is contrary to the density approach, where one generally has such a bound (see \cite{ChSh} or \cite{CGS}).  
\end{remark}

Next, we will turn to Lipschitz continuous test functions $h$. In contrast to bounded measurable test functions, there we will also be able to prove useful bounds for $\fnorm{g_h'}$. 

In order to obtain bounds for Lipschitz continuous test functions we need a further condition on the distribution $\mu$ which guarantees that its expected value exists.

\begin{cond}\label{gencond3}
The density $p$ is positive on the interval $(a,b)$ and absolutely continuous on every compact interval $[c,d]\subseteq(a,b)$. Furthermore, $E[|Z|]=\int_{\abquer}\abs{x}p(x)dx<\infty$.
\end{cond}

The following proposition, which is also proved in the appendix, includes bounds for both, $g_h$ and $g_h'$, when $h$ is Lipschitz.

\begin{prop}\label{genprop4}
Under Conditions \ref{gencond3} and \ref{gencond2} we have for any Lipschitz continuous test function 
$h:\abquer\rightarrow\R$ and any $x\in\abquer$:
\begin{enumerate}[{\normalfont (a)}]
\item $\abs{g_h(x)}\leq\fnorm{h'}\frac{F(x)E[Z]-\int_a^x yp(y)dy}{I(x)}$
\item $\abs{g_h'(x)}\leq\fnorm{h'}\frac{\int_a^xF(s)dsG(x)+\int_x^b(1-F(s))dsH(x)}{p(x)\eta(x)^2}$
\end{enumerate}
Here, for $x\in\abquer$ the positive functions $H(x)$ and $G(x)$ are defined by 
\[H(x):=I(x)-\ga(x)F(x)=p(x)\eta(x)-\ga(x)F(x)\text{ and } G(x):=H(x)+\ga(x)\,.\]
\end{prop}

\begin{remark}\label{genrem3}
 In general, the term $S(x):=\frac{F(x)E[Z]-\int_a^x yp(y)dy}{I(x)}$ cannot be bounded uniformly in $x\in(a,b)$ unless $\abs{\ga}$ grows at least linearly in $x$.
\end{remark}

\begin{cor}\label{gencor2}
Assume Condition \ref{gencond3} and that $\ga(x)=c(E[Z]-x)$ for some $c>0$. Then we have for any Lipschitz continuous test function $h:\abquer\rightarrow\R$ and each $x\in(a,b):$
\begin{enumerate}[{\normalfont (a)}]
 \item $\fnorm{g_h}\leq \frac{\fnorm{h'}}{c}$
 \item $\abs{g_h'(x)}\leq\frac{2\fnorm{h'}}{c}\frac{H(x)G(x)}{I(x)\eta(x)}
=2\fnorm{h'}\frac{\int_a^xF(s)ds\int_x^b(1-F(t))dt}{\eta(x)\bigl(E[Z]F(x)-\int_a^xyp(y)dy\bigr)}$

\end{enumerate}
\end{cor}

\begin{proof}
Claim (a) follows from Proposition \ref{genprop4} (a) and the observation that in this case we have 

\[I(x)=\int_a^x\ga(y)p(y)dy=c\int_a^x\bigl(E[Z]-y\bigr)p(y)dy=c\bigl(E[Z]F(x)-\int_a^x yp(y)dy\bigr)\,.\]

Part (b) follows from Proposition \ref{genprop4} (b) and Lemma \ref{genlemma1} by observing that in this case

\begin{eqnarray*}
H(x)&=&I(x)-\ga(x)F(x)=c\Biggl(\int_a^x\bigl(E[Z]-t\bigr)p(t)dt-\bigl(E[Z]-x\bigr)F(x)\Biggl)\\
&=&c\Biggl(E[Z]F(x)-\int_a^xtp(t)dt-E[Z]F(x)+xF(x)\Biggr)\\
&=&c\int_a^x F(s)ds
\end{eqnarray*} 

and, similarly, $G(x)=c\int_x^b(1-F(s))ds$.
\end{proof}

\begin{remark}\label{genrem4}
\begin{enumerate}[(i)]
 \item It is quite remarkable that in the case of normal approximation (via its classical Stein equation) the bound given in Corollary \ref{gencor2} (a) even improves on the best bound $2\fnorm{h'}$ currently mentioned in the literature (see, e.g. \cite{CGS} or \cite{CheShaSing}). In fact, in this case $c=1$ and thus our bound reduces to $\fnorm{h'}$.  

 \item For concrete distributions the ratio appearing in the bound for $g_h'(x)$ may be bounded uniformly in $x$ by some constant which can sometimes also be computed explicitely. Nevertheless, in \cite{EdViq} the authors give mild conditions for the existence of a finite constant $k$ such that $\fnorm{g_h'}\leq k\fnorm{h'}$ for any Lipschitz-continuous $h$. In practice, these conditions are usually met. However, there is no hope of estimating the constant $k$ by their method of proof. Thus, for concrete distributions and explicit constants it might therefore by useful to work with our bound from Corollary \ref{gencor2} (b).

\item For the normal distribution and also for the larger class of distributions discussed in \cite{EiLo10}, one also has a bound of the form $\fnorm{g_h''}\leq C\fnorm{h'}$ for some finite constant $C$ holding for each Lipschitz function $h$. As was shown by a universal example in \cite{EdViq}, such a bound cannot be expected unless $a=-\infty$ and $b=\infty$, if one takes $\ga(x)=c(E[Z]-x)$. This is why we will have to assume that $h'$ is also Lipschitz, for example by demanding that $h$ has two bounded derivatives. For the Stein solutions of the density approach, however, there are many examples of distributions, whose support is strictly included in $\R$ but for which such bounds are available (see, e.g., chapter 13 of \cite{CGS}).     
\end{enumerate}
\end{remark}

Next, we will discuss, how we can express the density $p$ of $\mu$ in terms of $\ga$ and $\eta$. This will be useful to bound the second derivative of $g_h$ in some special cases. Let $x_0$ be as in Condition \ref{gencond2}. Since $\eta'=\ga-\eta\psi$ and hence 
$\psi=\frac{\ga-\eta'}{\eta}$, we have

\begin{eqnarray}\label{dichteformel}
p(x)&=&p(x_0)\exp\Bigl(\int_{x_0}^x\psi(t)dt\Bigr)=p(x_0) \exp\Bigl(\int_{x_0}^x\frac{\ga(t)}{\eta(t)}dt\Bigr)\frac{\eta(x_0)}{\eta(x)}\\
&=&\frac{I(x_0)}{\eta(x)}\exp\Bigl(\int_{x_0}^x\frac{\ga(t)}{\eta(t)}dt\Bigr)\,.
\end{eqnarray}

Formula (\ref{dichteformel}) is a more general version of formula (3.14) in \cite{NouVie09} and is also derived in \cite{KuTu12}.
Now, differentiating Stein's equation (\ref{gensteineq}), we obtain for $h$ Lipschitz

\begin{equation}\label{gensecder}
\eta(x)g_h''(x)+g_h'(x)\bigl(\eta'(x)+\ga(x)\bigr)=h'(x)-\ga'(x)g_h(x)=:h_2(x)\,. 
\end{equation}

This means, that the function $\tilde{g}:=g_h'$ is a solution of the Stein equation corresponding to the test function $h_2$ for the distribution $\tilde{\mu}$ which satisfies the Stein identity

\[E\Bigl[\eta(Y)f'(Y)+\bigl(\eta'(Y)+\ga(Y)\bigr)f(Y)\Bigr]=0\,,\]

where $Y\sim \tilde{\mu}$. From (\ref{dichteformel}) we know that a density $\tilde{p}$ of $\tilde{\mu}$ is given by

\begin{equation}\label{formelptilde}
 \tilde{p}(x)=\frac{\tilde{K}}{\eta(x)}\exp\Bigl(\int_{x_0}^x\frac{\eta'(t)+\ga(t)}{\eta(t)}dt\Bigr)=K\exp\Bigl(\int_{x_0}^x\frac{\ga(t)}{\eta(t)}dt\Bigr)\,,
\end{equation}

where $\tilde{K},K>0$ are suitable normalizing constants. Thus, if we have bounds for the first derivative of the Stein solutions for the distribution $\tilde{\mu}$ and for Lipschitz functions $h$, then we obtain from this observation bounds on $g_h''$ for $h$ such that $h_2$ is Lipschitz (if $\ga(x)=c(E[Z]-x)$, this essentially means that $h'$ must be Lipschitz).   
Of course, it has to be verified, that $\tilde{\mu}(h_2)=0$ in order to use these bounds. But this will follow immediately, if one can show that $\tilde{g}$ belongs to a class of functions for which the Stein identity for $\tilde{\mu}$ is valid. \\

In the following we are going to adress the question of how to extend Stein's equation and its solution on $\R\setminus\abquer$. To this end, we suppose that the functions $\ga$ and $\eta$ both have a natural extension to all of $\R$, for example that they are given by some ``analytical expression'' on $\abquer$ which still makes sense on $\R\setminus\abquer$. 
We will need the following condition.

\begin{cond}\label{condout1}
The functions $\ga$ and $\eta$ are defined as before on $(a,b)$ and may be extended on $\R$ such that the following properties are satisfied:
\begin{enumerate}[(i)]
 \item On $\abquer$ the function $\ga$ has all the properties listed in Condition \ref{gencond1} and is continuous and strictly decreasing on $\R$.
\item The function $\eta$ is given by (\ref{etaformel}) for $x\in(a,b)$ and is absolutely continuous on every compact sub-interval $[c,d]\subseteq\R$. Furthermore, $\eta(x)\not=0$ for all $x\in\R\setminus\{a,b\}$.
\end{enumerate}
\end{cond}

Let $h:\R\rightarrow\R$ be a given Borel-measurable test function with $E\bigl[\abs{h(Z)}\bigr]<\infty$. Then we have the following Stein equation, valid now for all $x\in\R$: 

\begin{equation}\label{gensteineq2}
 \eta(x)g'(x)+\ga(x)g(x)=h(x)-\mu(h)=:\htilde(x)
\end{equation}

We already know, how to solve (\ref{gensteineq2}) for $x\in(a,b)$. So now, we will assume that at least one of $a$ and $b$ is finite and try to solve the equation outside $(a,b)$. Furthermore, we will discuss conditions that ensure that the composed solution $g_h$ behaves nicely at the edges $a$ and/or $b$. We will henceforth assume Condition \ref{condout1}.\\
For $x\not=a,b$ equation (\ref{gensteineq2}) is clearly equivalent to 

\begin{equation}\label{gensteineq3}
 g'(x)=-\frac{\ga(x)}{\eta(x)}g(x)+\frac{\htilde(x)}{\eta(x)}\,.
\end{equation}

Let us assume that both $a>-\infty$ and $b<\infty$ (the other cases are of course included) and let $F_l$ be any primitive function of $\frac{\ga}{\eta}$ on $(-\infty,a)$. Such a function exists by continuity and is hence continuously differentiable. 
By the method of variation of the constant one may derive the following formula for $x\in(-\infty,a)$:

\begin{equation}\label{solformleft}
g_h(x):= \exp\bigl(-F_l(x)\bigr)\int_a^x\frac{\htilde(t)}{\eta(t)}\exp\bigl(F_l(t)\bigr)dt\,,
\end{equation}
  
if this integral exists. Note that this property does not depend on the particular choice of the primitive function $F_l$.
For a fixed primitive function $F_l$ of $\frac{\ga}{\eta}$ on $(-\infty,a)$ we define the function 
\[q_l:=\frac{\exp\bigl(F_l\bigr)}{\eta}\,.\]

Analogously, for a given primitive function $F_r$ of $\frac{\ga}{\eta}$ on $(b,\infty)$ we define the function 
\[q_r:=\frac{\exp\bigl(F_r\bigr)}{\eta}\,.\] 

Note that inside the interval $(a,b)$ we have that $\log(\eta p)$ is a primitive of $\frac{\ga}{\eta}$ and hence $q_l$ plays the role of $p$ on $(-\infty,a)$ (and similarly for $q_r$). As we have observed, we will need the following Condition:

\begin{cond}\label{condout2}
There exist primitive functions $F_l:(-\infty,a)\rightarrow\R$ and $F_r:(b,\infty)\rightarrow\R$ of $\frac{\ga}{\eta}$ on $(-\infty,a)$ and on $(b,\infty)$, respectively, such that for each $x\in(-\infty,a)$ the function $q_l$ is integrable over $[x,a)$ and for each $y\in(b,\infty)$ the function $q_r$ is integrable over $(b,y]$. We may thus define functions $Q_l:(-\infty,a)\rightarrow\R$ by 
$Q_l(x):=\int_a^x q_l(t)dt$ and $Q_r:(b,\infty)\rightarrow\R$ by $Q_r(y):=\int_b^y q_r(t)dt$.
\end{cond}

Similarly, for $x\in(b,\infty)$ we arrive at the definition 

\begin{equation}\label{solformright}
g_h(x):= \exp\bigl(-F_r(x)\bigr)\int_b^x\frac{\htilde(t)}{\eta(t)}\exp\bigl(F_r(t)\bigr)dt
= \exp\bigl(-F_r(x)\bigr)\int_b^x \htilde(t)q_r(t)dt \,,
\end{equation}

if this integral exists. So, in the following we will always implicitly assume that also $\int_b^y \abs{\htilde(t)q_r(t)}dt<\infty$ and 
$\int_x^a \abs{\htilde(t)q_l(t)}dt<\infty$ both hold for each $x\in(-\infty,a)$ and each $y\in(b,\infty)$.

\begin{remark}\label{remout1}
 Note that the definition of the solution $g_h$ does not depend on the choice of the primitive functions $F_l$ and $F_r$ since two such functions may only differ by an additive constant.
\end{remark}

Next, we prove that the above constructed solution $g_h$ is continuous as long as $h$ is continuous at $a$ and $b$.
To deal with the limits $\lim_{x \nearrow a} g_h(x)$ and $\lim_{x\searrow b}g_h(x)$ we first formulate a condition which will usually be satisfied in practice.

\begin{cond}\label{condout3}
The functions $F_l$ and $F_r$ satisfy $\lim_{x \nearrow a} F_l(x)=\pm\infty$ and\\
 $\lim_{x\searrow b}F_r(x)=\pm\infty$.
\end{cond}

Again, the validity of this condition does not depend on the choice of the functions $F_l$ and $F_r$. 
By Condition \ref{condout3} we may again apply de l'H\^{o}pital's rule to compute

\begin{eqnarray}\label{conta2}
\lim_{x\nearrow a}g_h(x)&=& \lim_{x\nearrow a}\frac{\int_a^x\htilde(t)q_l(t)dt}{\exp\bigl(F_l(x)\bigr)}
=\lim_{x\nearrow a}\frac{\htilde(x)q_l(x)}{\exp\bigl(F_l(x)\bigr)F_l'(x)}\nonumber\\
&=&\lim_{x\nearrow a}\frac{\htilde(x)q_l(x)}{\ga(x)q_l(x)}
=\lim_{x\nearrow a}\frac{\htilde(x)}{\ga(x)}=\frac{1}{\ga(a)}\lim_{x\nearrow a}\htilde(x)
\end{eqnarray}

and, similarly, $\lim_{x\searrow b}g_h(x)=\frac{1}{\ga(b)}\lim_{x\searrow b}\htilde(x)$,
if these limits exist. Again, the continuity of $h$ at $a$ and $b$ is sufficient for this to hold. By Proposition \ref{propcont} we can thus conclude the following proposition.

\begin{prop}\label{propout1}
Assume Conditions \ref{condout1}, \ref{condout2} and \ref{condout3} and let $h:\R\rightarrow\R$ be a Borel-measurable test function satisfying the above integrability conditions and being continuous at $a$ and $b$. Then the Stein solution $g_h$ as defined above may be extended to a continuous function on $\R$ by letting 

\[g_h(a):=\frac{h(a)-\mu(h)}{\ga(a)}\quad\text{and}\quad g_h(b):=\frac{h(b)-\mu(h)}{\ga(b)}\,.\]

\end{prop}

Next, we want to present bounds on $g_h(x)$ and $g_h'(x)$ for $x\notin\abquer$. But first we will show that our conditions 
already imply that $\eta(x)<0$ if $x<a$ or $x>b$. Since $F_l'(x)=\frac{\ga(x)}{\eta(x)}$ is then negative on $(-\infty,a)$ this also ensues that $\lim_{x\nearrow a}F_l(x)=-\infty$. Similarly, $\lim_{x\searrow b} F_r(x)=-\infty$.

\begin{prop}\label{propout2}  
Assume Conditions \ref{condout1}, \ref{condout2} and \ref{condout3}. Then the functions $\eta$, $F_l$ and $F_r$ have the following properties:

\begin{enumerate}[{\normalfont (a)}]
\item For all $x\in\R\setminus\abquer$ we have $\eta(x)<0$. Especially, by continuity we have $\eta(a)=\eta(b)=0$.
\item We have $\lim_{x\nearrow a}F_l(x)=\lim_{x\searrow b} F_r(x)=-\infty$.
\end{enumerate}
\end{prop}

\begin{proof}
To prove (a), first note, that by Condition \ref{condout1} $\eta$ has no sign changes on $(-\infty,a)$. Suppose contrarily to the assertion, that $\eta(x)>0$ for all $x\in(-\infty,a)$. Then, the function $q_l$ is also positive and hence 
$0\leq\int_x^a q_l(t)dt=-Q_l(x)$ for each $x\in(-\infty,a)$. By Conditions \ref{condout2} and \ref{condout3} we may apply de l'H\^{o}pital's rule to conclude

\begin{eqnarray*}
0&\leq&\lim_{x\nearrow a}\frac{-Q_l(x)}{\exp\bigl(F_l(x)\bigr)}=\lim_{x\nearrow a}\frac{-q_l(x)}{\frac{\ga(x)}{\eta(x)}\exp\bigl(F_l(x)\bigr)}=\lim_{x\nearrow a}\frac{-q_l(x)}{\ga(x)q_l(x)}\\
&=&-\lim_{x\nearrow a}\frac{1}{\ga(x)}=\frac{-1}{\ga(a)}<0\,,
\end{eqnarray*} 

by Condition \ref{condout1}. This is a contradiction and hence we must have $\eta(x)<0$ for $x\in(-\infty,a)$.
Similarly, one shows that also $\eta(x)<0$ for $x\in(b,\infty)$.\\
To prove (b), note that $F_l'(x)=\frac{\ga(x)}{\eta(x)}<0$ for $x<a$. By Condition \ref{condout3} this necessarily implies that $\lim_{x\nearrow a}F_l(x)=-\infty$. Analogously, we have $F_r'(x)>0$ for $x>b$ (since $\ga(x)<0$ there) which by Condition \ref{condout3} implies that $\lim_{x\searrow b} F_r(x)=-\infty$. 
\end{proof}

Proposition \ref{propout2} particularly implies the conclusion of Proposition \ref{etaprop} and hence makes Condition \ref{condboundsupp} redundant, at least as far as the assertion of this proposition is concerned. 
In order to get general bounds, we will need yet another condition on the functions $F_l$ and $F_r$

\begin{cond}\label{condout4}
The functions $F_l$ and $F_r$ satisfy $\lim_{x\searrow -\infty} F_l(x)=\lim_{x\nearrow \infty}F_r(x)=+\infty$. 
\end{cond}

The next result gives bounds on $g_h$ for bounded, Borel-measurable functions $h$. As usual, the proof is in the appendix.

\begin{prop}\label{propout3}
Assume Conditions \ref{condout1}, \ref{condout2}, \ref{condout3} and \ref{condout4} and let $m$ be a median for $\mu$. Then, for any bounded and Borel-measurable test function $h:\R\rightarrow\R$ we have 

\[\fnorm{g_h}\leq\fnorm{h-\mu(h)}\max\Biggl(\frac{1}{2I(m)},\,\frac{1}{\ga(a)},\,\frac{-1}{\ga(b)}\Biggr)\,.\]

\end{prop}

Before turning to Lipschitz test functions, we dicuss properties of the functions $\exp(F_l)$ and $\exp(F_r)$, respectively. In particular, we will show, that they correspond to the function $I$ on $(a,b)$ and have similar integral representations.

\begin{prop}\label{propIout} 
We define the functions $I_l:(-\infty,a)\rightarrow\R$ and\\ $I_r:(b,\infty)\rightarrow\R$ by $I_l(x):=\exp\bigl(F_l(x)\bigr)$ and 
$I_r(x):=\exp\bigl(F_r(x)\bigr)$. Then the following representations hold:

\begin{enumerate}[{\normalfont (a)}]
 \item For each $x\in(-\infty,a)$ we have $I_l(x)=\int_a^x \ga(t)q_l(t)dt$.
 \item For each $x\in(b,\infty)$ we have $I_r(x)=\int_b^x \ga(t)q_r(t)dt$.
\end{enumerate}
In particular, these integrals exist. 
\end{prop}

\begin{proof}
We only prove (a). Let $x\in(-\infty,a)$ and let $(a_n)_{n\in\N}$ be any sequence in $(-\infty,a)$ with $\lim_{n\to \infty} a_n=a$. Then, for each $n\in\N$

\begin{eqnarray*}
\int_{a_n}^x\ga(t)q_l(t)dt&=&\int_ {a_n}^x\frac{\ga(t)}{\eta(t)}\exp\bigl(F_l(t)\bigr)dt
=\int_ {a_n}^xF_l'(t)\exp\bigl(F_l(t)\bigr)dt\\
&=&\int_{F_l(a_n)}^{F_l(x)} e^s ds=\exp\bigl(F_l(x)\bigr)-\exp\bigl(F_l(a_n)\bigr)\\
&\stackrel{n\to\infty}{\longrightarrow}&\exp\bigl(F_l(x)\bigr)\,,
\end{eqnarray*}
 
by Proposition \ref{propout2}.
Hence $\int_a^x\ga(t)q_l(t)dt$ exists and equals $I_l(x)$. 
\end{proof}

\begin{prop}\label{propout4}
Assume Conditions \ref{gencond3}, \ref{condout1}, \ref{condout2}, \ref{condout3} and \ref{condout4}. 
Then, for any Lipschitz-continuous test function $h:\R\rightarrow\R$ the following bounds hold true: 
\begin{enumerate}[{\normalfont (a)}]

\item For each $x\in(-\infty,a)$ we have $\abs{g_h(x)}\leq\fnorm{h'}\frac{Q_l(x)E[Z]-\int_a^x sq_l(s)ds}{I_l(x)}$.
\item For each $x\in(b,\infty)$ we have  $\abs{g_h(x)}\leq\fnorm{h'}\frac{Q_r(x)E[Z]-\int_b^x sq_r(s)ds}{I_r(x)}$.

\item For each $x\in(-\infty,a)$ we have

\begin{eqnarray*}
\abs{g_h'(x)}&\leq&\frac{\fnorm{h'}}{-\eta(x)I_l(x)}\Biggl(\ga(x)\Bigl(-xQ_l(x)+\int_a^xtq_l(t)dt\Bigr)\\
&&+\bigl(E[Z]-x\bigr) \Bigl(Q_l(x)\ga(x)-I_l(x)\Bigr)\Biggr)
\end{eqnarray*}

\item For each $x\in(b,\infty)$ we have

\begin{eqnarray*}
\abs{g_h'(x)}&\leq&\frac{\fnorm{h'}}{-\eta(x) I_r(x)}\Biggl(\ga(x)\Bigl(xQ_r(x)-\int_b^xtq_r(t)dt\Bigr)\\
&&+\bigl(x-E[Z]\bigr)\Bigl(\ga(x)Q_r(x)-I_r(x)\Bigr)\Biggr) 
\end{eqnarray*}

\end{enumerate}
\end{prop}

\begin{cor}\label{outcor1}
Assume Conditions \ref{gencond3}, \ref{condout1}, \ref{condout2}, \ref{condout3} and \ref{condout4} and that 
$\ga(x)=c\bigl(E[Z]-x\bigr)$ for some $c>0$. Then, for any Lipschitz-continuous test function $h:\R\rightarrow\R$ one has the following bounds:
\begin{enumerate}[{\normalfont (a)}]
\item $\fnorm{g_h}\leq\frac{\fnorm{h'}}{c}$
\item For each $x\in(-\infty,a)$ we have

\[\abs{g_h'(x)}\leq2\fnorm{h'}\frac{\ga(x)\Bigl(-xQ_l(x)+\int_a^xtq_l(t)dt\Bigr)}{-\eta(x)I_l(x)}\,.\]

\item For each $x\in(b,\infty)$ we have

\[\abs{g_h'(x)}\leq2\fnorm{h'}\frac{\ga(x)\Bigl(xQ_r(x)-\int_b^xtq_r(t)dt\Bigr)}{-\eta(x)I_l(x)}\,.\]

\end{enumerate}
\end{cor}
Again, the proof is in the appendix. \\
Next, we will give a short account of the properties of the solutions $g_z$ corresponding to the test functions 
$h_z:=1_{(-\infty,z]}$, $z\in\R$, yielding the Kolmogorov distance. In general, these functions are only of interest for 
$z\in(a,b)$ since for any real-valued random variable $W$ and any $z\leq a$ we have

\[\Bigl|P(W\leq z)-P(Z\leq z)\Bigr|=P(W\leq z)\]

and for any $z\geq b$ we have
\[\Bigl|P(W\leq z)-P(Z\leq z)\Bigr|=1-P(W\leq z)=P(W>z)\,.\]

These probabilities have nothing to do with the distribution $\mu$ of $Z$ and hence, one should be able to bound them as well directly for a given $W$. Thus, we will focus on $z\in(a,b)$. 

\begin{prop}\label{genpropgz}
Let $z\in(a,b)$ be given. Then, under the Conditions \ref{gencond1}, \ref{condout1}, \ref{condout2} and \ref{condout3} we have:

\[g_z(x)=
\begin{cases}
\frac{(1-F(z))Q_l(x)}{I_l(x)}&,\,x< a\\
\frac{1-F(z)}{\ga(a)}&,\,x=a\\
\frac{F(x)(1-F(z))}{I(x)}&,\,a<x\leq z\\
\frac{F(z)(1-F(x))}{I(x)}&,\,z<x<b\\
\frac{-F(z)}{\ga(b)}&,\, x=b\\
\frac{-F(z)Q_r(x)}{I_r(x)}&,\, x>b
\end{cases}
\]

Furthermore, one has the bounds $\fnorm{g_z}=\frac{F(z)\bigl(1-F(z)\bigr)}{I(z)}$ and also\\ 
$\sup_{z\in(a,b)}\fnorm{g_z}<\infty$.
The functions $g_z$ are absolutely continuous on every compact sub-interval of $\R$ and

\[g_z'(x)=
\begin{cases}
\frac{-(1-F(z))q_l(x)(\ga(x)Q_l(x)-I_l(x))}{I_l(x)^2}&,\,x< a\\
\frac{(1-F(z))p(x)(I(x)-\ga(x)F(x))}{I(x)^2}&,\,a<x<z\\
\frac{-F(z)p(x)(I(x)+\ga(x)(1-F(x)))}{I(x)^2}&,\,z<x<b\\
\frac{-F(z)q_r(x)(I_r(x)-\ga(x)Q_r(x))}{I_r(x)^2}&,\, x>b
\end{cases}
\]

\end{prop}

\begin{remark}\label{outrem2}
\begin{enumerate}[(i)]
\item It is clear, that a similar discussion of the solutions $g_z$ is possible, if $a=-\infty$ or $b=\infty$. 
\item Note, that we can write $g_z'(x)=\frac{(1-F(z))p(x)H(x)}{I(x)^2}$ for $x\in(a,z)$ and\\ 
$g_z'(x)=\frac{-F(z)p(x)G(x)}{I(x)^2}$ for $x\in(z,b)$, with the functions $H$ and $G$ from Proposition \ref{genprop4}. For concrete distributions one may often prove, that $g_z'$ is increasing on $(a,z)$ and decreasing on $(z,b)$, but this seems to be hard to prove in generality, if it is true at all.   
\end{enumerate}
\end{remark}

Finally, in our general setting, we will prove suitable ``plug-in theorems'' for exchangeable pairs satisfying our general regression property (\ref{genreg}). As was observed in \cite{Ro08} for the normal distribution, in case of univariate distributional approximations, one does not need the full strength of exchangeability, but equality in distribution of the random variables $W$ and $W'$ is sufficient. This may allow for a greater choice of admissible couplings in several situations, or at least, relaxes the verification of asserted properties.\\

In the following, let $(\Om,\A,P)$ be a probability space and let $W,W'$ be real-valued random variables defined on this space such that $W\stackrel{\D}{=}W'$. Let, as before, $\mu$ be our target distribution with support $\abquer$ fulfilling Condition \ref{gencond1} . From now on we will assume, that the random variables $W$ and $W'$ only have values in an interval $I\supseteq(a,b)$ where both functions $\eta$ and $\ga$ are defined (recall that it might be the case that $\eta$ can only be defined on $(a,b)$).


\begin{prop}\label{genpluginprop1}
Assume, that $\ga$ satisfies Condition \ref{gencond2} and that $W$ is square integrable with $E[\abs{\ga(W)}]<\infty$. Furthermore, let the pair $(W,W')$ satisfy the general regression property (\ref{genreg}). Let $f:I\rightarrow\R$ be an absolutely continuous function with $\fnorm{f},\,\fnorm{f'}<\infty$, where $f'$ is a given, Borel-measurable version of the derivative of $f$. Then we have

\begin{eqnarray}\label{genplugin1}
&&\Bigl|E\bigl[\eta(W)f'(W)+\ga(W)f(W)\bigr]\Bigr|\nonumber\\
&\leq&\fnorm{f'}E\Bigl[\bigl|\eta(W)-\frac{1}{2\la}E\bigl[(W'-W)^2|W\bigr]\bigr|\Bigr]\nonumber\\
&&+\frac{1}{\la}E\Bigl[(W'-W)^2\int_0^1(1-s)\bigl|f'(W)-f'\bigl((W+s(W'-W)\bigr)\bigr|ds\Bigr]\nonumber\\
&&+\frac{\fnorm{f}}{\la}E\bigl[\abs{R}\bigr]\,.
\end{eqnarray}

If $f'$ is also absolutely continuous and $\fnorm{f''}<\infty$ for some Borel-measurable version $f''$ of the second derivative, then we also have the bound 

\begin{eqnarray}\label{genplugin2}
&&\Bigl|E\bigl[\eta(W)f'(W)+\ga(W)f(W)\bigr]\Bigr|\nonumber\\
&\leq&\fnorm{f'}E\Bigl[\bigl|\eta(W)-\frac{1}{2\la}E\bigl[(W'-W)^2|W\bigr]\bigr|\Bigr]\nonumber\\
&&+\frac{\fnorm{f''}}{6\la}E\bigl[\abs{W'-W}^3\bigr]\nonumber\\
&&+\frac{\fnorm{f}}{\la}E\bigl[\abs{R}\bigr]\,.
\end{eqnarray}

\end{prop}

\begin{proof}
Let $x_0$ be as in Condition \ref{gencond2} and define the function $G:I\rightarrow\R$ by $G(x):=\int_{x_0}^xf(y)dy$. Then, by a suitable version of Taylor's formula, for each $x,x'\in I$ we have

\begin{eqnarray*}
G(x')-G(x)&=&G'(x)(x'-x)+\int_x^{x'}(x'-t)G''(t)dt\\
&=&f(x)(x'-x)+\int_x^{x'}(x'-t)f'(t)dt\\
&=&f(x)(x'-x)+(x'-x)^2\int_0^1(1-s)f'\bigl(x+s(x'-x)\bigr)ds\,.
\end{eqnarray*}

Hence, by distributional equality, we obtain

\begin{eqnarray*}
&&0=E\bigl[G(W')\bigr]-E\bigl[G(W)\bigr]\\
&=&E\bigr[f(W)(W'-W)\bigr]+E\Bigl[(W'-W)^2\int_0^1(1-s)f'\bigl((W+s(W'-W)\bigr)ds\Bigr]\\
&=&E\Bigr[f(W)E\bigl[W'-W|W\bigr]\Bigr]+E\Bigl[(W'-W)^2\int_0^1(1-s)f'\bigl((W+s(W'-W)\bigr)ds\Bigr]\\
&=&\la E\bigl[f(W)\ga(W)\bigr]+E\bigl[f(W)R\bigr]\\
&&+E\Bigl[(W'-W)^2\int_0^1(1-s)f'\bigl((W+s(W'-W)\bigr)ds\Bigr]\,,
\end{eqnarray*}

yielding 

\begin{equation}\label{genplugin1eq1}
E\bigl[f(W)\ga(W)\bigr]=-\frac{1}{\la}E\Bigl[(W'-W)^2\int_0^1(1-s)f'\bigl((W+s(W'-W)\bigr)ds\Bigr]-\frac{1}{\la}E\bigl[f(W)R\bigr].
\end{equation}

This immediately implies the identity

\begin{eqnarray}\label{genplugin1eq2}
&&E\bigl[\eta(W)f'(W)+\ga(W)f(W)\bigr]\nonumber\\
&=&E\Bigr[f'(W)\bigl(\eta(W)-\frac{1}{2\la}(W'-W)^2\bigr)\Bigr]\nonumber\\
&&+\frac{1}{\la}E\Bigl[(W'-W)^2\int_0^1(1-s)\Bigl(f'(W)-f'\bigl((W+s(W'-W)\bigr)\Bigr)ds\Bigr]\nonumber\\
&&-\frac{1}{\la}E\bigl[f(W)R\bigr]\,.
\end{eqnarray}

From (\ref{genplugin1eq2}) and the assumptions on $f$ the bound (\ref{genplugin1}) now easily follows. To prove 
(\ref{genplugin2}) it suffices to observe that 

\[\Bigl| f'(W)-f'\bigl((W+s(W'-W)\bigr)\Bigr|\leq\fnorm{f''}\, s\abs{W'-W}\]

and $\int_0^1s(1-s)ds=\frac{1}{6}$.
\end{proof}

\begin{remark}
\begin{enumerate}[(i)]
\item From the first term on the right hand side of (\ref{genplugin1}) we see, that the bound can only be useful, if 
$E\bigl[(W'-W)^2|W\bigr]\approx 2\la \eta(W)$. Similarly, the third term reveals, that, indeed, $R$ should be of smaller order than $\la$.

\item The proof shows, that Proposition \ref{genpluginprop1} can easily be generalized to the situation, where there is a $\sigma$-algebra $\F$ with 
$\sigma(W)\subseteq\F\subseteq\A$ and the more general regression property 

\begin{equation}\label{genreg2}
E\bigl[W'-W|\F\bigr]=\la\ga(W)+R
\end{equation}
with some $\F$-measurable remainder term $R$ is satisfied. 

\item If $\calH$ is some class of test functions, such that there are finite, positive constants $c_0$, $c_1$ and $c_2$ with
$\fnorm{g_h}\leq c_0$, $\fnorm{g_h'}\leq c_1$ and $\fnorm{g_h''}\leq c_2$ for each $h\in\calH$, then (\ref{genplugin2}) 
immediately yields a bound on the distance

\[d_{\calH}\bigl(\mu,\calL(W)\bigr)=\sup_{h\in\calH}\Bigl|E\bigl[h(W)\bigr]-E\bigl[h(Z)\bigr]\Bigr|\,.\] 

\end{enumerate}
\end{remark}

\section{Stein's method for Beta distributions}\label{mainresultsbeta}

In this section we will specialize the abstract theory from section \ref{mainresultsgen} to the family of Beta distributions. 
This is not the first time, that Stein's method for Beta distributions is considered. In \cite{Sch01} the Stein equation is found, but no bounds on the solutions are given. Specialized to the Beta distributions, the general results from the papers \cite{KuTu12} and \cite{EdViq} yield useful bounds on the Stein solutions and their first derivatives. Our bounds are qualitatively comparable to those from \cite{EdViq}.\\ 
The class of Beta distributions is rather large including as special cases the arcsine distribution, the semi-circle law and the uniform distribution over an interval. Usually, the family of Beta distributions is defined as a family of distributions on $[0,1]$ depending on two parameters $a,b>0$. Recall, that for $a,b>0$, the \textit{Beta function} $B(a,b):=\int_0^1x^{a-1}(1-x)^{b-1}dx$ is well-defined and hence, the function  $q_{a,b}(x):=\frac{1}{B(a,b)}x^{a-1}(1-x)^{b-1}1_{(0,1)}(x)$ is a probability density function. The corresponding distribution $\nu_{a,b}$ is called the \textit{Beta distribution} on $[0,1]$ corresponding to the parameters $a$ and $b$. \\
However, in some cases, including the semicircular law, it is more convenient to consider the corresponding distribution on $[-1,1]$. Using the substitution rule, it is a matter of routine to check, that for $\al,\be>-1$ the function $p_{\alpha,\beta}(x):=C(\al,\be)(1-x)^\al(1+x)^\be 1_{(-1,1)}(x)$ is also a probability density function, where $C(\al,\be):=\frac{1}{B(\al+1,\be+1)2^{\al+\be+1}}=\frac{\Gamma(\al+\be+2)}{2^{\al+\be+1}\Gamma(\al+1)\Gamma(\be+1)}$. Here , we have used the well-known relation $B(a,b)=\frac{\Gamma(a)\Gamma(b)}{\Gamma(a+b)}$ between the Beta function and \textit{Euler's Gamma function}. The  probability distribution corresponding to the density $p_{\al,\be}$ will be denoted by $\mu_{\al,\be}$ and is called the \textit{Beta distribution} on $[-1,1]$ corresponding to the parameters $\al$ and $\be$. We denote its distribution function by $F\albe$. Thus, $F\albe(x)=\int_{-\infty}^x p\albe(t)dt$, $x\in\R$. One can easily see that if $X\sim \nu_{a,b}$ and if we define 
$Y:=2X-1$, then $Y\sim\mu_{b-1,a-1}$. Conversely, if $Y\sim\mu_{\al,\be}$ and if $X:=\frac{Y+1}{2}$, then $X\sim\nu_{\be+1,\al+1}$. As this transformation is merely a matter of translation and scaling it will only be necessary to develop the theory for one of these intervals. For definiteness we choose the interval $[-1,1]$ but will later on transfer our theory to the corresponding distributions on $[0,1]$.\\

From now on, fix $\al,\be>-1$. Now, we introduce a Stein identity for the Beta distribution $\mu_{\al,\be}$. It is easily checked, that its density $p_{\al,\be}$ satisfies the ordinary differential equation 

\begin{equation}\label{DGLDichteBeta}
0=(x^2-1)p_{\al,\be}'(x)-(\al+\be)xp_{\al,\be}(x)+(\be-\al)p_{\al,\be}(x)=:Af(x)
\end{equation} 

Integrating by parts one obtains the conjugate operator $L:=A^{*}$, which is defined by the equation 
$<Af,g>_{L^2}=<f,A^*g>_{L^2}$, and which is known to serve as a characterizing operator for the distribution $\mu_{\al,\be}$. To be concrete, in our case we have 

 \[Lg(x)=(1-x^2)g'(x)-(\al+\be+2)xg(x)+(\be-\al)g(x)\]

for smooth enough functions $g$, yielding the Stein identity 

\[E\Bigl[(1-Z^2)g'(Z)-\bigl[(\al+\be+2)Z+(\al-\be)\bigl]g(Z)\Bigr]=0\,,\]

where $Z\sim\mu_{\al,\be}$.

Next we want to show that the operator $L$ characterizes the distribution $\mu_{\al,\be}$. To do so, we introduce the function $\rho_{\al,\be}(x):=C(\al,\be)(1-x)^{\al+1}(1+x)^{\be+1}1_{(-1,1)}(x)=(1-x^2)p_{\al,\be}(x)$ and the class $\K_{\al,\be}$ consisting of all continuous and piecewise continuously differentiable functions $g:\R\rightarrow\R$ vanishing at infinity with\\
 $\int_\R |g'(x)|\rho_{\al,\be}(x)dx<\infty$. 
Our observations lead to the following proposition.

\begin{prop}[Stein characterization for $\mu_{\al,\be}$]\label{charbeta}
A real-valued random variable is distributed according to $\mu_{\al,\be}$ if and only if for all functions $g\in\K_{\al,\be}$ the expected values $E[(1-X^2)g'(X)]$ and $E[(\al+\be+2)Xg(X)+(\al-\be)g(X)]$ exist and coincide.
\end{prop}

\begin{proof}
First, let $\calL(X)=\mu_{\al,\be}$ and let $g\in\K_{\al,\be}$. By the hypothesis and the transformation formula, we have
\begin{eqnarray*}
\infty&>&\int_\R|g'(x)|\rho_{\al,\be}(x)dx=C(\al,\be)\int_{-1}^1|g'(x)|(1-x)^{\al+1}(1+x)^{\be+1}dx\\
&=&C(\al,\be)\int_{-1}^1(1-x^2)|g'(x)|(1-x)^{\al}(1+x)^{\be}dx=\int_\R|(1-x^2)g'(x)|p_{\al,\be}(x)dx\\
&=&E[|(1-X^2)g'(X)|]\,.
\end{eqnarray*}
Hence the expexted value $E[(1-X^2)g'(X)]$ exists. Since $g$ is continuous, it is bounded on $[-1,1]$ and so the expected value 
$E[(\al+\be+2)Xg(X)+(\al-\be)g(X)]$ exists, too. Again, by the transformation rule and since $g$ and $\rho_{\al,\be}$ are absolutely continuous on $[-1,1]$ we can use integration by parts and have

\begin{eqnarray*}
&&E[(1-X^2)g'(X)]\\
&=&\int_\R(1-x^2)g'(x)p_{\al,\be}(x)dx=\int_{-1}^1 g'(x)\rho_{\al,\be}(x)dx\\
&=&g\rho_{\al,\be}|_{-1}^1-\int_{-1}^1 g(x)\rho_{\al,\be}'(x)dx\\
&=&0-C(\al,\be)\int_{-1}^1 g(x)[(\be+1)(1+x)^\be(1-x)^{\al+1}-(\al+1)(1-x)^\al(1+x)^{\be+1}]dx\\
&=&-\int_{-1}^1 g(x)[\be-\al-x(\be+\al+2)]C(\al,\be)(1-x)^\al(1+x)^\be dx\\
&=&\int_\R g(x)[\al-\be+x(\al+\be+2)]p_{\al,\be}(x)dx\\
&=&E[(\al+\be+2)Xg(X)+(\al-\be)g(X)]\,.
\end{eqnarray*}

For the converse fix an arbitrary $z\in\R\setminus\{-1,+1\}$ and consider the solution $g_z$ to Stein's equation 

\begin{equation}\label{steingleichungbetaindikator}
(1-x^2)g'(x)-(\al+\be+2)xg(x)+(\be-\al)g(x)=h_z(x)-\mu_{\al,\be}\left((-\infty,z]\right)\,,
\end{equation}

where $h_z:=1_{(-\infty,z]}$. It will  be shown in Proposition \ref{gzk} below that $g_z\in\K_{\al,\be}$, so that by hypothesis we have 

\begin{eqnarray*}
0&=&E[(1-X^2)g_z'(X)-(\al+\be+2)Xg(X)+(\be-\al)g(X)]\\
&=& E[h_z(X)-\mu_{\al,\be}\left((-\infty,z]\right)]\\
&=&P(X\leq z)-\mu_{\al,\be}\left((-\infty,z]\right)\,.
\end{eqnarray*}

Since $z\in\R\setminus\{-1,+1\}$ was arbitrary and by continuity from the right of distribution functions, the proof is complete.
\end{proof}

Since we have fixed the parameters $\al$ and $\be$, henceforth we may and will suppress them as sub-indices at objects which might well depend on them (for example we will simply write $p$ for $p_{\al,\be}$ and so on). As we would like to use the theory from section \ref{mainresultsgen} we have to make sure, that our Stein identity for the Beta distribution fits into this framework, i.e. that relation (\ref{etaformel}) is satisfied with $\eta(x)=1-x^2$ and 

\[\ga(x)=-(\al+\be+2)x+(\be-\al)=-[(\al+\be+2)x+(\al-\be)]=(\al+\be+2)(E[Z]-x)\,,\]

where we have used, that $E[Z]=\frac{\be-\al}{\al+\be+2}$. In principle, this is clear, because we have just established a Stein characterization for $\mu$ and given the density $p$ and the function $\ga$, the corresponding 
$\eta$ is, of course, unique. However, we give a formal proof. 

According to (\ref{etaformel}) we must show that

\begin{equation}\label{vergaeta}
(1-x^2)p(x)=\int_{-1}^x [-(\al+\be+2)t+(\be-\al)]p(t)dt
\end{equation}

holds for all $x\in(-1,1)$. First note, that 

\[\psi(x)=\frac{p'(x)}{p(x)}=\frac{-(\al+\be)x+\be-\al}{1-x^2}\,.\]

 Differentiating the left hand side of (\ref{vergaeta}), we obtain

\begin{eqnarray*}
\frac{d}{dx}(1-x^2)p(x)&=&-2xp(x)+(1-x^2)p'(x)=p(x)\bigl(-2x+(1-x^2)\psi(x)\bigr)\\
&=&p(x)[-(\al+\be+2)x+(\be-\al)]
\end{eqnarray*}

which is of course the derivative of the right hand side, too. Since 

\[\lim_{x\searrow-1}(1-x^2)p(x)=0=\lim_{x\searrow-1}\int_{-1}^x [-(\al+\be+2)t+(\be-\al)]p(t)dt\]

relation (\ref{vergaeta}) is proved. Note, that we may extend the functions $\gamma$ and $\eta$ to functions on $\R$ by just the same ``analytical'' expressions as above. Next we will show that all the conditions from section \ref{mainresultsgen} are satisfied in the special case of Beta distributions.  It is easy to see, that Conditions \ref{gencond1}, \ref{gencond2}, \ref{gencond2} and  \ref{condout1}
are satisfied by the functions $\ga$, $\eta$ and $p$. 



Condition \ref{condboundsupp} is  also satisfied but need not be proved, because its most important conclusion, namely that $\eta(1)=\eta(-1)=0$ is clear from the above discussion. To verify Conditions \ref{condout2}, \ref{condout3} and \ref{condout4}, we must first define the functions $F_l$ on $(-\infty,-1)$ and $F_r$ on $(1,\infty)$. We claim, that the functions 

\[F_l(x):=(\al+1)\log(1-x)+(\be+1)\log(-1-x)\,,\quad x<-1\]

and

\[F_r(x):= (\al+1)\log(x-1)+(\be+1)\log(1+x)\,,\quad x>1\]

do the job. In fact, 

\[F_l'(x)=\frac{\al+1}{1-x}(-1)+\frac{\be+1}{-1-x}(-1)=\frac{-(\al+\be+2)x+\be-\al}{1-x^2}=\frac{\ga(x)}{\eta(x)}\]

for each $x<-1$ and similarly $F_r'(x)=\frac{\ga(x)}{\eta(x)}$ for $x>1$. From these two functions we obtain the functions 
$q_l:(-\infty,-1)\rightarrow\R$ and $q_r:(1,\infty)\rightarrow\R$ defined by

\[q_l(x):=\frac{\exp(F_l(x))}{\eta(x)}=\frac{(1-x)^{\al+1}(-1-x)^{\be+1}}{1-x^2}=(1-x)^{\al}(-1-x)^{\be}\]

and

\[q_r(x):=\frac{\exp(F_r(x))}{\eta(x)}=\frac{(x-1)^{\al+1}(1+x)^{\be+1}}{1-x^2}=(x-1)^{\al}(1+x)^{\be}\,.\]

These two functions are of course locally integrable and hence, Condition \ref{condout2} is satisfied. Since 

\[\lim_{x\nearrow-1}F_l(x)=-\infty=\lim_{x\searrow 1}F_r(x)\]

and 

\[\lim_{x\to-\infty}F_l(x)=+\infty=\lim_{x\to\infty}F_r(x)\,,\]

Conditions \ref{condout3} and \ref{condout4} also hold. Consequently, all results from section \ref{mainresultsgen} are valid in particular for the case of Beta distributions. 

Now, let $h:\R\rightarrow\R$ be a given Borel-measurable test function with\\ $\int_\R \abs{h(x)}p(x)dx<\infty$, 
$\int_{x}^{-1} \abs{h(x)}(-q_l(x))dx<\infty$ for each $x<-1$ and\\ $\int_{1}^x \abs{h(x)}(-q_r(x))dx<\infty$ for each $x>1$ and let $\htilde=h-\mu(h)$. Consider the corresponding Stein equation

\begin{equation}\label{steineqbeta}
(1-x^2)g'(x)-(\al+\be+2)xg(x)+(\be-\al)g(x)=h(x)-\mu_{\al,\be}(h)=:\htilde(x)\,.
\end{equation}

 Then, according to the theory from section \ref{mainresultsgen}, the Stein solution corresponding to \eqref{steineqbeta} is given by the function $g_h:\R\rightarrow\R$ with 

\begin{eqnarray}\label{ghformelbeta}
g_h(x)= 
\begin{cases}
 \exp\bigl(-F_l(x)\bigr)\int_{-1}^x\htilde(t)q_l(t)dt\,,& x<-1\\
\frac{h(-1)-\mu(h)}{2\be+2}\,,&x=-1\\
\frac{1}{(1-x^2)p(x)}\int_{-1}^x\htilde(t)p(t)dt\,,&-1<x<1\\
\frac{h(1)-\mu(h)}{2\al+2}\,,&x=1\\
\exp\bigl(-F_r(x)\bigr)\int_{1}^x\htilde(t)q_r(t)dt\,,& x<-1\\
\end{cases}
\end{eqnarray}

Here, the values of $g_h$ at $\pm1$ are arbitrary, but they are chosen such that $g_h$ is continuous, whenever $h$ is continuous at $\pm1$. This follows immediately from Proposition \ref{propout1}.\\

The next result, which is also proved in the appendix, completes the proof of Proposition \ref{charbeta} by showing that the solution $g_z$ is in the class $\K_{\albe}$ whenever $z\not=\pm1$. 

\begin{prop}\label{gzk}
For each $z\in\R\setminus\{-1,+1\}$ the solution $g_z$ belongs to the class $\K_{\albe}$, which was defined above before Proposition \ref{charbeta}. 
\end{prop}

Next, we will derive some results for the solutions $g_h$ from corresponding results in Section \ref{mainresultsgen}.

\begin{prop}\label{propbeta1}
Let $h:\R\rightarrow\R$ be bounded and Borel-measurable and let $m$ be a median for $\mu$. Then, we have the bound: 

\[\fnorm{g_h}\leq\fnorm{h-\mu(h)}\max\Biggl(\frac{1}{2(1-m^2)p(m)},\,\frac{1}{2\be+2},\,\frac{1}{2\al+2}\Biggr)\] 

\end{prop}

\begin{proof}
Since $\ga(-1)=2\be+2$ and $\ga(1)=-2\al-2$, this immediately follows from Proposition \ref{propout3}. 
\end{proof}

\begin{prop}\label{propbeta2}
Let $h:\R\rightarrow\R$ be Lipschitz-continuous.  Then, we have the following bounds: 
\begin{enumerate}[{\normalfont (a)}]
 \item $\fnorm{g_h}\leq\frac{\fnorm{h'}}{\al+\be+2}$
 \item There exists a constant $K_1$, only depending on $\al$ and $\beta$ such that \\
  $\fnorm{g_h'}\leq K_1\fnorm{h'}$. 
\end{enumerate}
\end{prop}
The proof is in the appendix.

Now, consider a twice differentiable function $h:\R\rightarrow\R$ with bounded first and second derivative. Recall the discussion following Remark \ref{genrem4}. It easily follows from \eqref{formelptilde} that, in the case of Beta distributions, we have $\tilde{\mu}_{\al,\be}=\mu_{\al+1,\be+1}$ and we have to show that the function $\tilde{g}:=g_h'$ satisfies the Stein identity for $\mu_{\al+1,\be+1}$, i.e. for $Y\sim\mu_{\al+1,\be+1}$ we have 

\begin{equation}\label{secder1}
E\Bigl[(1-Y)^2\tilde{g}'(Y)+\bigl[-(\al+\be+4)Y+\be-\al\bigr]\tilde{g}(Y)\Bigr]=0\,.
\end{equation}

The following lemma, which is proved in the appendix, will be useful.

\begin{lemma}\label{betalemma2}
For the Beta distribution $\mu_{\al,\be}$ and a given bounded, Borel-measurable function $u:\R\rightarrow\R$, the Stein equation 

\begin{equation}\label{steineqnc}
 \eta(x)f'(x) +\ga(x)f(x)=u(x)
\end{equation}
has a bounded solution $f$ on $(-1,1)$ if and only if $E[u(Z)]=0$.  
\end{lemma}

Now, from Proposition \ref{propbeta2} (b) we know that $\tilde{g}$ is bounded and from (\ref{gensecder}) we know that 
$\tilde{g}$ satisfies the Stein equation corresponding to $\tilde{\mu}$ for the test function\\ $h_2(x)=h'(x)-\ga'(x)g_h(x)$.
By Lemma \ref{betalemma2} we thus have that\\ $\int_\R h_2(x)d\mu_{\al+1,\be+1}(x)=0$. Hence, $\tilde{g}$ must be the only bounded solution to the Stein equation for $\mu_{\al+1,\be+1}$ corresponding to the test function $h_2$. Since 

\[h_2(x)=h'(x)-\ga'(x)g_h(x)=h'(x)+(\al+\be+2)g_h(x)\]

we have

\[h_2'(x)=h''(x)+(\al+\be+2)g_h'(x)\]

and from Proposition \ref{propbeta2} we see that $h_2$ is Lipschitz with minimal Lipschitz constant 

\[\fnorm{h_2'}\leq\fnorm{h''}+(\al+\be+2)K_1(\al,\be)\fnorm{h'}\]

where the constant $K_1(\al,\be)$ from Proposition \ref{propbeta2} only depends on $\al$ and $\be$. Applying Proposition \ref{propbeta2} again, this time to the distribution $\mu_{\al+1,\be+1}$ and the Stein solution $\tilde{g}$, we obtain 

\begin{eqnarray*}
\fnorm{g_h''}&=&\fnorm{\tilde{g}'}\leq K_1(\al+1,\be+1)\fnorm{h_2'}\\
&\leq& K_1(\al+1,\be+1)\bigl(\fnorm{h''}+(\al+\be+2)K_1(\al,\be)\fnorm{h'}\bigr)\,.
\end{eqnarray*}

Hence, there is a constant $K_2$ depending only on $\al$ and $\be$ such that 

\[\fnorm{g_h''}\leq K_2(\fnorm{h'}+\fnorm{h''})\]

for all twice differentiable functions $h$ with bounded first and second derivative.
We have thus proved the following proposition.

\begin{prop}\label{propbeta3}
There exists a finite constant $K_2$ depending only on $\al$ and $\be$ such that for each twice differentiable function $h:\R\rightarrow\R$ with bounded first and second derivative we have the bound

\[\fnorm{g_h''}\leq K_2(\fnorm{h'}+\fnorm{h''})\,.\]

\end{prop}

Now we are in the position to provide a ``plug-in theorem'' for the Beta approximation using exchangeable pairs.

\begin{theorem}\label{betaplugin1}
Let $W,W'$ be identically distributed, real-valued random variables on a common probabilty space $(\Om,\A,P)$ satisfying the regression property 

\[E\bigl[W'-W|W\bigr]=\la\ga(W)+R\]

for some constant $\la>0$ and a random variable $R$. Then for each twice differentiable function $h$ with bounded first and second derivative and with $E\bigl[\abs{h(W)}\bigr]<\infty$ we have the bound 

\begin{eqnarray*}
&&\bigl|E[h(W)]-\mu_{\al,\be}(h)\bigr|\\
&\leq&K_1\fnorm{h'}E\Bigl[\bigl|\eta(W)-\frac{1}{2\la}E\bigl[(W'-W)^2|W\bigr]\bigr|\Bigr]\nonumber\\
&&+\frac{K_2\bigl(\fnorm{h'}+\fnorm{h''}\bigr)}{6\la}E\bigl[\abs{W'-W}^3\bigr]\nonumber\\
&&+\frac{\fnorm{h'}}{(\al+\be+2)\la}E\bigl[\abs{R}\bigr]\,,
\end{eqnarray*}

where the constants $K_1$ and $K_2$ are from Propositions \ref{propbeta2} and \ref{propbeta3}, respectively.
\end{theorem}

\begin{proof}
This immediately follows from Propositions \ref{genpluginprop1}, \ref{propbeta2}, \ref{propbeta3} and since $g_h$ is a solution to Stein's equation (\ref{gensteineq}).
\end{proof}

In the following we will transfer the developed theory to the Beta distributions $\nu_{a,b}$ on $[0,1]$. We start with the Stein identity for $\nu_{a,b}$, where $a,b>0$ are fixed parameters. Let $X\sim\nu_{a,b}$, then $Y:=2X-1\sim\mu_{b-1,a-1}$ and hence for each smooth enough function $f$ we have

\[0=E\bigl[(1-Y^2)f'(Y)-(a+b)Yf(Y)+(a-b)f(Y)\bigr]\,.\]

Let $\tilde{f}(x):=f(2x-1)$. Then $\tilde{f}'(x)=2f'(2x-1)$ and $\tilde{f}(X)=f(Y)$. Hence, we obtain

\begin{eqnarray*}
0&=&E\bigl[(1-Y^2)f'(Y)-(a+b)Yf(Y)+(a-b)f(Y)\bigr]\\
&=&E\bigl[4X(1-X)f'(2X-1)-(a+b)(2X-1)f(2X-1)+(a-b)f(2X-1)\bigr]\\
&=&E\bigl[2X(1-X)\tilde{f}'(X)-2(a+b)X\tilde{f}(X)+2a\tilde{f}(X)\bigr]\\
&=&2E\bigl[X(1-X)\tilde{f}'(X)-(a+b)X\tilde{f}(X)+a\tilde{f}(X)\bigr]\,.
\end{eqnarray*} 

So, a Stein identity for $X\sim\nu_{a,b}$ is given by 

\[E\Bigl[X(1-X)f'(X)+\bigl[-(a+b)X+a\bigr]f(X)\Bigr]=0\]

for all smooth enough functions $f:\R\rightarrow\R$. Hence, for $\nu_{a,b}$ the functions $\eta$ and $\ga$ are given by 
$\eta(x)=x(1-x)$ and $\ga(x)=-(a+b)x+a=(a+b)\bigl(\frac{a}{a+b}-x\bigr)$. Note, that $E[X]=\frac{a}{a+b}$ if 
$X\sim\nu_{a,b}$. Having derived the Stein identity for $\nu_{a,b}$, the Stein equation corresponding to a Borel-measurable test function $h:\R\rightarrow\R$ with $\int_\R \abs{h(x)}d\nu_{a,b}(x)<\infty$ is given by

\begin{equation}\label{steineqnu}
x(1-x)f'(x)+\bigl[-(a+b)x+a\bigr]f(x)=h(x)-\nu_{a,b}(h)=:\hat{h}(x)\,.
\end{equation}

Let a test function $h:\R\rightarrow\R$ with $\int_\R \abs{h(x)}d\nu_{a,b}(x)<\infty$ be given and let 
$h_1(y):=h\bigl(\frac{y+1}{2}\bigr)$. Consider the above constructed solution $g$ to the Stein equation for $\mu_{b-1,a-1}$ corresponding to the test function $h_1$. In the folllowing let the real variables $x$ and $y$ be related by $y:=2x-1$ or 
$x:=\frac{y+1}{2}$. Letting $f(x):=2g(2x-1)$ and noting $f'(x)=4g(y)$ we obtain

\begin{eqnarray*}
&&x(1-x)f'(x)+\bigl[-(a+b)x+a\bigr]f(x)\\
&=&\frac{y+1}{2}\bigl(1-\frac{y+1}{2}\bigr)4g'(y)+\bigl[-(a+b)(y+1)+2a\bigr]g(y)\\
&=&(1-y^2)g'(y)+\bigl[-(a+b)y+a-b\bigr]g(y)\\
&=&h_1(y)-\mu_{b-1,a-1}(h_1)=h(x)-\nu_{a,b}(h)\,,
\end{eqnarray*} 

since for any admissible function $u:\R\rightarrow\R$ we have $E[u(X)]=E[u_1(Y)]$ where 
$u_1(y):=u\bigl(\frac{y+1}{2}\bigr)$. Hence, $f$ is a solution to (\ref{steineqnu}). 
Thus, we immediately get bounds on the solutions of the Stein equation for $\nu_{a,b}$ from our above developed theory. 
In the following, we will always denote by $f_h$ the Stein solution to (\ref{steineqnu}) which is constructed in the explained way. 

\begin{prop}\label{propbeta4}
Let $h:\R\rightarrow\R$ be Borel-measurable with $\int_\R\abs{h(x)}d\nu_{a,b}(x)<\infty$. 
\begin{enumerate}[{\normalfont (a)}]
\item If $h$ is bounded, then 
$\fnorm{f_h}\leq \fnorm{h-\nu_{a,b}(h)}\max\Bigl(\frac{1}{2m(1-m)q_{a,b}(m)},\,\frac{1}{a},\,\frac{1}{b}\Bigr)$, where $m$ is a median for $\nu_{a,b}$.
\item If $h$ is Lipschitz, then $\fnorm{f_h}\leq\frac{2}{a+b}\fnorm{h'}$ and $\fnorm{f_h'}\leq C_1\fnorm{h'}$, where $C_1$ only depends on $a$ and $b$.
\item If $h$ is twice differentiable with bounded first and second derivative, then\\ 
$\fnorm{f_h''}\leq C_2\bigl(\fnorm{h'}+\fnorm{h''}\bigr)$, where $C_2$ only depends on $a$ and $b$. 
\end{enumerate}
\end{prop}

\begin{proof}
Claim (a) follows from Proposition \ref{propout3}. Since $f_h(x)=2g_h(2x-1)$ for all $x\in\R$, (b) and (c) follow from Propositions \ref{propbeta2} and \ref{propbeta3}.
\end{proof}

Now, let $V,V'$ be identically distributed, real-valued random variables on a common probability space $(\Om,\A,P)$. For the approximation of $\calL(V)$ by $\nu_{a,b}$ the general regression property from Section \ref{mainresultsgen} is

\begin{equation}\label{regbeta2}
 E\bigl[V'-V|V\bigr]=\la(a+b)\bigl(\frac{a}{a+b}-V\bigr)+R\,,
\end{equation}

where, again, $\la>0$ is constant and $R$ is a hopefully small remainder term. For the distribution $\nu_{a,b}$ Theorem \ref{betaplugin1} becomes the following:

\begin{theorem}\label{betaplugin2}
Let $V,V'$ be identically distributed, real-valued random variables on a common probability space $(\Om,\A,P)$ satisfying equation (\ref{regbeta2}). Then, for each twice differentiable function $h:\R\rightarrow\R$ with bounded first and second derivative and with 
$E\bigl[\abs{h(V)}\bigr]<\infty$ we have the bound 

\begin{eqnarray*}
&&\bigl|E[h(V)]-\nu_{a,b}(h)\bigr|\\
&\leq&C_1\fnorm{h'}E\Bigl[\bigl|V(1-V)-\frac{1}{2\la}E\bigl[(V'-V)^2|V\bigr]\bigr|\Bigr]\\
&&+\frac{C_2\bigl(\fnorm{h'}+\fnorm{h''}\bigr)}{6\la}E\bigl[\abs{V'-V}^3\bigr]\\
&&+\frac{2\fnorm{h'}}{(a+b)\la}E\bigl[\abs{R}\bigr]\,,
\end{eqnarray*}

where the constants $C_1$ and $C_2$ are those from Proposition \ref{propbeta4}.
\end{theorem}

\begin{proof}
 The assertion is clear from Propositions \ref{genpluginprop1} and \ref{propbeta4} and since $f_h$ is a solution to Stein's equation (\ref{steineqnu}).
\end{proof}

\section{Application to the Polya urn model}\label{applications}

In this section we prove a quantitative version of the fact that the relative number of drawn red balls in a Polya urn model converges in distribution to a suitable Beta distribution, if the number of total drawings tends to infinity. This model will serve as an application of our Stein method for the Beta distribution, as developed in section \ref{mainresultsbeta}. We start by introducing the stochastic model:\\

Imagine an urn containing at the beginning $r$ red and $w$ white balls and fix an integer $c>0$. At each time point $n\in\N$ a ball is drawn at random from the urn, its color is noticed and this ball together with $c$ further balls of the same color is replaced to the urn. 
For each $n\in\N$ let $X_n$ be the indicator random variable of the event that the $n$-th drawn ball is a red one. Then 
$S_n:=\sum_{j=1}^n X_j$ denotes the total number of drawn red balls among the first $n$ drawings. It is a well-known fact from elementary probability theory that for each $n\in\N$ and all $x_1,\ldots,x_n\in\{0,1\}$ we have  

\[P(X_1=x_1,\ldots,X_n=x_n)=\frac{\prod_{i=0}^{k-1}(r+ci)\prod_{j=0}^{n-k-1}(w+cj)}{\prod_{l=0}^{n-1}(s+w+cl)}\,,\]

where $k:=\sum_{j=1}^n x_j$. This shows particularly that the sequence $(X_j)_{j\in\N}$ is exchangeable.

It now follows, that for each $k=0,\ldots,n$ we have 

\[P(S_n=k)=\binom{n}{k}\frac{\prod_{i=0}^{k-1}(r+ci)\prod_{j=0}^{n-k-1}(w+cj)}{\prod_{l=0}^{n-1}(s+w+cl)}\,, \]

or, with $a:=\frac{r}{c}$ and $b:=\frac{w}{c}$,

\[P(S_n=k)=\frac{\binom{-a}{k}\binom{-b}{n-k}}{\binom{-a-b}{n}}\,.\]

The distribution of $S_n$ is usually referred to as the \textit{Polya distribution} with parameters $n$, $a$ and $b$. It is a well-known fact that the distribution of $\frac{1}{n}S_n$ converges weakly to the distribution $\nu_{a,b}$ as $n$ goes to infinity, where the Beta distribution $\nu_{a,b}$ was defined in section \ref{mainresultsbeta}. A convenient way to prove this weak convergence result is to use the formula 

\begin{equation}\label{defin1}
 P(S_n=k)=\int_0^1 b(k; n,p)d\nu_{a,b}(p)\,,
\end{equation}

together with the weak law of large numbers for Bernoulli random variables to deal with the binomial probabilities $b(k;n,p)=\binom{n}{k}p^k(1-p)^{n-k}$.

Formula (\ref{defin1}) can be proved by a straight-forward computation using the relations $B(a+1,b)=\frac{a}{a+b}B(a,b)$ and $B(a,b)=B(b,a)$ for the Beta function, where $a,b>0$, and can also be viewed as a consequence of a special instance of \textit{de Finetti's representation theorem} for infinite exchangeable sequences. Note, however, that one generally does not know the corresponding mixing measure from de Finetti's theorem and hence, identity (\ref{defin1}) is not a direct consequence of this theorem.\\

From now on, we will present a Stein's method proof of the above distributional convergence result and, as usual, also derive a rate of convergence. We will usually suppress the time index $n$ and let $V:=V_n:=\frac{1}{n}S_n$ denote the random variable of interest. For the construction of the exchangeable pair, we use the well-known \textit{Gibbs sampling} procedure with the slight simplification, that due to exchangeability of $X_1,\ldots,X_n$ we need not choose at random the index of the summand from $S_n$, which has to be replaced. Instead, we will always replace $X_n$ by $X_n'$, which is constructed as follows:\\

Observe $X_1=x_1,\ldots,X_n=x_n$ and construct $X_n'$ according to the distribution $\calL(X_n|X_1=x_1,\ldots,X_{n-1}=x_{n-1})$. Then, letting $V':=V_n':=V-\frac{1}{n}X_n+\frac{1}{n}X_n'$, the pair $(V,V')$ is exchangeable. In order to use Stein's method of exchangeable pairs, we need to establish a suitable regression property. This is the content of the following proposition.

\begin{prop}\label{polyaprop1}
The exchangeable pair $(V,V')$ satisfies the regression property

\[E\bigl[V'-V|V\bigr]=\frac{a+b}{n(a+b+n-1)}\Bigl(\frac{a}{a+b}-V\Bigr)=\la\ga_{a,b}(V)\,,\]
 
where $\ga_{a,b}(x)=(a+b)\bigl(\frac{a}{a+b}-x\bigr)$ and $\la=\la_n= \frac{1}{n(a+b+n-1)}$.

\end{prop}

\begin{proof}
We have $V'-V=\frac{X_n'}{n}- \frac{X_n}{n}$ and by exchangeability of $X_1,\ldots,X_n$ it clearly holds that 
$E[X_n|V]=E[X_n|S_n]=\frac{1}{n}S_n=V$. Also, by the definition of $X_n'$ and since $X_n'$ only assumes the values $0$ and $1$ we have 
for any $x_1,\ldots,x_{n-1}\in\{0,1\}$

\begin{eqnarray*}
&&E[X_n'|X_1=x_1,\ldots,X_n=x_n]=E[X_n|X_1=x_1,\ldots,X_{n-1}=x_{n-1}]\\
&=&P(X_n=1|X_1=x_1,\ldots,X_{n-1}=x_{n-1})=\frac{r+c\sum_{j=1}^{n-1}x_j}{r+w+c(n-1)}\,,
\end{eqnarray*}

and hence,

\begin{eqnarray*}
 E[X_n'|X_1,\ldots,X_n]=\frac{r+c\sum_{j=1}^{n-1}X_j}{r+w+c(n-1)}=\frac{r+cnV-cX_n}{r+w+c(n-1)}\,.
\end{eqnarray*}

Thus, since $\sigma(V)\subseteq\sigma(X_1,\ldots,X_n)$, we obtain

\begin{eqnarray*}
E[X_n'|V]&=&E\Bigl[E\bigl[X_n'|X_1,\ldots,X_n\bigr]\,|V\Bigr]\\
&=&\frac{r+cnV-cV}{r+w+c(n-1)}=\frac{r+c(n-1)V}{r+w+c(n-1)}\\
&=&\frac{a+(n-1)V}{a+b+n-1}\,.
\end{eqnarray*}

Finally, we have

\begin{eqnarray*}
E[V'-V|V]&=&\frac{1}{n}E[X_n'-X_n|V]=\frac{1}{n}\frac{a+(n-1)V}{a+b+n-1}-\frac{1}{n}V\\
&=& \frac{a-(a+b)V}{n(a+b+n-1)}=\frac{a+b}{n(a+b+n-1)}\Bigl(\frac{a}{a+b}-V\Bigr)\,,
\end{eqnarray*}

as was to be shown.
\end{proof}

Next, we will compute the quantity $E\bigl[(V'-V)^2|V\bigr]$.

\begin{prop}\label{polyaprop2}
We have for the above constructed exchangeable pair $(V,V')$
\[E\bigl[(V'-V)^2|V\bigr]= \frac{1}{n^2(a+b+n-1)}\Bigl((2n+b-a)V-2nV^2+a\Bigr)\]
and hence
\[\frac{1}{2\la}E\bigl[(V'-V)^2|V\bigr]=V(1-V)+\frac{b-a}{2n}V+\frac{a}{2n}\,.\]
 
\end{prop}

\begin{proof}
From the general theory of Gibbs sampling (see the author's PhD thesis, to appear) it is known, that

\begin{eqnarray*}\label{polyaeq1}
&& E\bigl[(V'-V)^2|V\bigr]\\
&=&\frac{1}{n^2}\Bigl(E[X_n|V]+E\bigl[E[X_n^2|X_1,\ldots,X_{n-1}]\,|\,V\bigr]
-2E\bigl[X_nE[X_n|X_1,\ldots,X_{n-1}]|V\bigr]\Bigr)\,.\nonumber
\end{eqnarray*}

Since $X_n^2=X_n$ we have from the proof of Proposition \ref{polyaprop1} that

\[E[X_n^2|X_1,\ldots,X_{n-1}]=E[X_n|X_1,\ldots,X_{n-1}]=\frac{a+nV-X_n}{a+b+n-1}\,,\]

and hence

\[E\bigl[E[X_n^2|X_1,\ldots,X_{n-1}]\,|\,V\bigr]=\frac{a+(n-1)V}{a+b+n-1}\,,\]

where we have used $E[X_n|V]=V$ again. Finally, we compute 

\begin{eqnarray*}
 &&E\bigl[X_nE[X_n|X_1,\ldots,X_{n-1}]|V\bigr]=E\Bigl[\frac{aX_n+nVX_n-X_n^2}{a+b+n-1}\,\bigl|\,V\Bigr]\\
&=&\frac{aV+nV^2-V}{a+b+n-1}=\frac{(a-1)V+nV^2}{a+b+n-1}\,.
\end{eqnarray*}

Putting pieces together, we eventually obtain

\begin{eqnarray*}
 E\bigl[(V'-V)^2|V\bigr]&=&\frac{1}{n^2}\Bigl(V+\frac{a+(n-1)V}{a+b+n-1}-2\frac{(a-1)V+nV^2}{a+b+n-1}\Bigr)\\
&=&\frac{1}{n^2(a+b+n-1)}\Bigl((2n+b-a)V-2nV^2+a\Bigr)\,.
\end{eqnarray*}

The last assertion easily follows from this and from $\la= \frac{1}{n(a+b+n-1)}$.
\end{proof}

Recall that for the distribution $\nu_{a,b}$ we have $\eta(x):=\eta_{a,b}(x)=x(1-x)$ and hence, we obtain from Proposition 
\ref{polyaprop2} that

\begin{equation}\label{polyaeq2}
E\Bigl[\bigl|\eta(V)-\frac{1}{2\la}E\bigl[(V'-V)^2|V\bigr]\bigr|\Bigr]=E\Bigl[\bigl|\frac{a-b}{2n}V-\frac{a}{2n}\bigr|\Bigl] 
\leq\frac{\abs{a-b}+a}{2n}\,,
\end{equation}

since $\abs{V}\leq1$.
Similarly, since $\abs{V'-V}=\frac{1}{n}\abs{X_n'-X_n}\leq\frac{1}{n}$ we have 

\begin{eqnarray}\label{polyaeq3}
\frac{1}{6\la}E\bigl[\abs{V'-V}^3\bigr]&\leq&\frac{n(a+b+n-1)}{6}\frac{1}{n^3}=\frac{a+b+n-1}{6n^2}\nonumber\\
&=&\frac{1}{6n}+\frac{a+b-1}{6n^2}=O\Bigl(\frac{1}{n}\Bigr)\,. 
\end{eqnarray}

From Theorem \ref{betaplugin2} we can now conclude the following result.

\begin{theorem}\label{polyarate}
For each twice differentiable function $h:\R\rightarrow\R$ with bounded first and second derivative we have 

\begin{eqnarray*}
&&\bigl|E[h(V)]-\nu_{a,b}(h)\bigr|\\
&\leq&\Bigl(C_1\fnorm{h'}\frac{\abs{a-b}+a}{2}+C_2\bigl(\fnorm{h'}+\fnorm{h''}\bigr)
\bigl(\frac{1}{6}+\frac{a+b-1}{6n}\bigr)\Bigr)\frac{1}{n}\\
&=&O\Bigl(\frac{1}{n}\Bigr)\,,
\end{eqnarray*}

with the constants $C_1$ and $C_2$ from Proposition \ref{propbeta4}.
\end{theorem}

\begin{proof}
Since $V$ assumes only values in $[0,1]$, the condition $E\bigl[\abs{h(V)}\bigr]<\infty$ from Theorem \ref{betaplugin2} is trivially met. 
The assertion now follows immediately from Theorem \ref{betaplugin2}, (\ref{polyaeq2}) and (\ref{polyaeq3}).
\end{proof}

\begin{remark}
\begin{enumerate}[(a)]
\item In \cite{GolRei12} the authors use a different technique within Stein's method for the Beta distributions, which compares the Stein characterization of the target distribution with that of the approximating discrete distribution, to prove that, in the special case $c=1$, the convergence rate of order $n^{-1}$ from Theorem \ref{polyarate} even holds in the Wasserstein distance and they compute an explicit constant in the bound. They also show that the rate of convergence is optimal. Using their technique and the bounds from Proposition \ref{propbeta4} one can easily see that the rate of order $n^{-1}$ in the Wasserstein distance also holds in the case $c\geq2$. However, in order to obtain an explicit constant, some further work has to be done to bound the constant $C_1$ from Proposition \ref{propbeta4} in the case that one of the values $a,b$ is strictly smaller than one. 
\item In \cite{FulGol10} the authors use the zero bias coupling within Stein's method for normal approximation to prove bounds on the distance of a normalized version of the quantity $V$ to the standard normal distribution. In particular, they show that a CLT holds whenever the parameters $n$, $a$ and $b$ tend to infinity in a suitable fashion.
\end{enumerate}
\end{remark}

\section{Appendix}\label{appendix}
In this section we provide the proofs of some of the results from Sections \ref{mainresultsgen} and \ref{mainresultsbeta} and state and prove some further auxiliary results, which are only used within proofs.\\

\subsection{A general version of de l'H\^{o}pital's rule} 
The following result justifies all our calculations, which invoke de l'H\^{o}pital's rule. Its proof is suppressed for reasons of space, but will be given in the author's PhD thesis. 

\begin{theorem}[Generalization of one of de l'H\^{o}pital's rules]\label{Hospital}
Let $a<b$ be extended real numbers and let $f,g:(a,b)\rightarrow\R$ be functions with the following properties:
\begin{enumerate}[{\normalfont (i)}]
 \item If $a<a'<b'<b$, then both, $f$ and $g$, are absolutely continuous on $[a',b']$.
 \item We have $g'(x)>0$ for $\lambda$-almost all $x\in(a,b)$ and with $E:=\{x\in(a,b)\,:\, g'(x)\not=0\}$ it holds that 
$\lim_{n\to\infty}\frac{f'(x_n)}{g'(x_n)}=\delta\in\R$ for each sequence $(x_n)_{n\in\N}\in E$ converging to $a$.
\end{enumerate}
If $\lim_{x\searrow a}f(x)= \lim_{x\searrow a}g(x)=0$, then $g(x)\not=0$ for all $x\in(a,b)$ and
\[\lim_{x\searrow a}\frac{f(x)}{g(x)}=\delta\,.\]

The same conclusion holds if $g'(x)<0$ for almost all $x\in(a,b)$ and an analogous result is true for $\lim_{x\nearrow b}$.
\end{theorem}

\subsection{Proofs from Section \ref{mainresultsgen}}
\begin{proof}[Proof of Proposition \ref{genprop3}]
With $\htilde=h-\mu(h)$, since $I=\eta\cdot p$, we have for $a<x<b$

\[\abs{g_h(x)}=\frac{\abs{\int_a^x \htilde(t)p(t)dt}}{\abs{p(x)\eta(x)}}=\frac{\abs{\int_a^x \htilde(t)p(t)dt}}{I(x)}\leq\fnorm{\htilde}\frac{F(x)}{I(x)}\]

Let $M:(a,b)\rightarrow\R$ be given by $M(x):=\frac{F(x)}{I(x)}$. By l'H\^{o}pital's rule (see Theorem \ref{Hospital}) we have 

\[\lim_{x\searrow a}M(x)=\lim_{x\searrow a}\frac{p(x)}{\ga(x)p(x)}=\lim_{x\searrow a}\frac{1}{\ga(x)}=\frac{1}{\lim_{x\searrow a}\ga(x)}\]

which exists in $[0,\infty)$ by Condition \ref{gencond2}. Here, we used the convention $\frac{1}{\infty}=0$.

Moreover,

\[\lim_{x\nearrow b}M(x)=\frac{1}{\lim_{x\nearrow b}I(x)}=+\infty\]

again by Condition \ref{gencond2} and by Proposition \ref{genprop1}. Furthermore, we have 

\begin{eqnarray*}
M'(x)=\frac{p(x)I(x)-p(x)\ga(x)F(x)}{I(x)^2}=\frac{p(x)}{I(x)^2}\Bigl(I(x)-\ga(x)F(x)\Bigr)>0 
\end{eqnarray*}

for each $x\in(a,b)$ since by the positivity of $p$ and because $\ga$ is strictly decreasing 

\[I(x)=\int_a^x\ga(t)p(t)dt>\ga(x)\int_a^xp(t)dt=\ga(x)F(x)\,.\]

Hence, $M$ is strictly increasing and thus for each $x\in(a,m]$:

\[\abs{g_h(x)}\leq \fnorm{\htilde}\frac{F(m)}{I(m)}=\frac{\fnorm{h-\mu(h)}}{2I(m)}\,.\]

The same bound can be proved for $x\in(m,b)$ by using the representation 

\[g_h(x)=-\frac{1}{I(x)}\int_x^b(h(t)-\mu(h))p(t)dt\]

and the fact that also $1-F(m)=\frac{1}{2}$.

\end{proof}

The following two well-known lemmas will be needed for the proof of Proposition \ref{genprop4}. Their proofs are included only for reasons of completeness.

\begin{lemma}\label{genlemma1}
 Let $-\infty\leq a<b\leq\infty$ and let $\mu$ be a probability measure (not necessarily absolutely continuous with respect to $\la$) with $\supp(\mu)\subseteq\abquer$. Let $F$ be the distribution function corresponding to $\mu$ and suppose that $\int_a^b\abs{x}d\mu(x)<\infty$. Then, for each $x\in\abquer$ we have
\begin{enumerate}[{\normalfont (a)}]
 \item $\int_a^xF(t)dt=xF(x)-\int_{\overline{(a,x]}}s d\mu(s)$
 \item $\int_x^b (1-F(t))dt=\int_{(x,\infty)} sd\mu(s)-x(1-F(x))$
\end{enumerate}

\end{lemma}

\begin{proof}
By Fubini's theorem we have

\begin{eqnarray*}
\int_a^xF(t)dt&=&=\int_{-\infty}^xF(t)dt=\int_{(-\infty,x]}F(t)dt
=\int_{(-\infty,x]}\Bigl(\int_{(-\infty,t]}d\mu(s)\Bigr)dt\\
&=&\int_{(-\infty,x]}\Bigl(\int_{[s,x]}dt\Bigr)d\mu(s)=\int_{(-\infty,x]}(x-s)d\mu(s)\\
&=&xF(x)-\int_{(-\infty,x]}sd\mu(s)\\
&=&xF(x)-\int_{\overline{(a,x]}}s d\mu(s)\,.
\end{eqnarray*}
This proves (a). Similarly, we have

\begin{eqnarray*}
\int_x^b(1-F(t))dt&=&\int_x^\infty(1-F(t))dt=\int_{(x,\infty)}\mu\bigl((t,\infty)\bigr)dt\\
&=&\int_{(x,\infty)}\int_{(t,\infty)}d\mu(s)dt=\int_{(x,\infty)}\int_{(x,s)}dtd\mu(s)\\
&=&\int_{(x,\infty)}(s-x)d\mu(s)\\
&=&\int_{(x,\infty)} sd\mu(s)-x(1-F(x))\,,
\end{eqnarray*}
 
proving (b).
\end{proof}

\begin{lemma}\label{liplemma}
Let $-\infty\leq a<b\leq\infty$ and let $\mu$ be a probability measure (not necessarily absolutely continuous with respect to $\la$) with $\supp(\mu)\subseteq\abquer$. Let $F$ be the distribution function corresponding to $\mu$, let $Z\sim\mu$ and let $h:\abquer\rightarrow\R$ be Lipschitz continuous with $E[|h(Z)|]<\infty$. Then the following assertions hold true:
\begin{enumerate}[{\normalfont (a)}]
\item For each $y\in\R$ we have\\
 $h(y)-\mu(h)=\int_{-\infty}^yF(s)h'(s)ds-\int_y^\infty(1-F(s))h'(s)ds$.  

\item For each $x\in\abquer$ we have\\
$\int_{(a,x]}(h(y)-\mu(h))d\mu(y)=-(1-F(x))\int_a^xF(s)h'(s)ds-F(x)\int_x^b(1-F(s))h'(s)ds$.

\end{enumerate}
\end{lemma}

\begin{proof}
Since $\mu$ is a probability measure 
we have by the fundamental theorem of calculus for Lebesgue integration and by Fubini's theorem

\begin{eqnarray*}
h(y)-\mu(h)&=&\int_\R\left(h(y)-h(t)\right)d\mu(t)=\int_{\R}\Bigl(\int_t^y h'(s)ds\Bigr)d\mu(t)\\
&=&\int_{(-\infty,y]}\Bigl(\int_{[t,y]}h'(s)ds\Bigr)d\mu(t)-\int_{(y,\infty)}\Bigl(\int_{[y,t)}h'(s)ds\Bigr)d\mu(t)\\
&=&\int_{(-\infty,y]}\Bigl(\int_{(-\infty,s]}d\mu(t)\Bigr)h'(s)ds 
-\int_{(y,\infty)}\Bigl(\int_{(s,\infty)}d\mu(t)\Bigr)h'(s)ds\\
&=&\int_{(-\infty,y]}F(s)h'(s)ds-\int_{(y,\infty)}(1-F(s))h'(s)ds\\
&=&\int_{-\infty}^y F(s)h'(s)ds-\int_y^{\infty}(1-F(s))h'(s)ds\,.
\end{eqnarray*}
This proves (a). As to (b), we have using (a) and its proof

\begin{eqnarray*}
 &&\int_{(a,x]}(h(y)-\mu(h))d\mu(y)=\int_{(-\infty,x]}(h(y)-\mu(h))d\mu(y)\\
&=&\int_{(-\infty,x]}\Bigl(\int_{(-\infty,y]}F(s)h'(s)ds-\int_{(y,\infty)}(1-F(s))h'(s)ds\Bigr)d\mu(y)\\
&=&\int_{(-\infty,x]}\int_{(-\infty,y)}F(s)h'(s)ds d\mu(y)
-\int_{(-\infty,x]}\int_{[y,\infty)}(1-F(s))h'(s)ds d\mu(y)\\
&=&\int_{(-\infty,x)}F(s)h'(s)\Bigl(\int_{(s,x]}d\mu(y)\Bigr)ds
-\int_\R(1-F(s))h'(s)\Bigl(\int_{(-\infty,x\wedge s]}d\mu(y)\Bigr)ds\\
&=&\int_{(-\infty,x)}F(s)h'(s)(F(x)-F(s))ds - \int_{(-\infty,x]}F(s)(1-F(s))h'(s)ds\\
&&-\int_{(x,\infty)}F(x)(1-F(s))h'(s)ds\\
&=&\int_{-\infty}^x\Bigl(F(s)F(x)h'(s)-F(s)^2 h'(s)+F(s)^2h'(s)-F(s)h'(s)\Bigr)ds\\
&&-F(x)\int_x^\infty(1-F(s))h'(s)ds\\
&=&-(1-F(x))\int_{-\infty}^xF(s)h'(s)ds - F(x)\int_x^\infty(1-F(s))h'(s)ds\\
&=&-(1-F(x))\int_{a}^xF(s)h'(s)ds - F(x)\int_x^b(1-F(s))h'(s)ds\,,
\end{eqnarray*}

as claimed.
\end{proof}

\begin{proof}[Proof of Proposition \ref{genprop4}]
First, we prove (a). Recall the representation 
\[g_h(x)=\frac{1}{I(x)}\int_a^x(h(y)-\mu(h))p(y)dy=\frac{1}{I(x)}\int_{(a,x]}(h(y)-\mu(h))d\mu(y)\,.\]

By Lemmas \ref{liplemma} and \ref{genlemma1} we thus obtain that

\begin{eqnarray*}
&&\abs{I(x)g_h(x)}\\
&=&\left|-(1-F(x))\int_a^xF(s)h'(s)ds-F(x)\int_x^b(1-F(s))h'(s)ds\right|\\
&\leq&\fnorm{h'}\Biggl((1-F(x))\int_a^x F(s)ds+F(x)\int_x^b(1-F(s))ds\Biggr)\\
&=&\fnorm{h'}\Biggl((1-F(x))\Bigl(xF(x)-\int_a^x sp(s)ds\Bigr)+F(x)\Bigl(-x(1-F(x))+\int_x^b sp(s)ds\Bigr)\Biggr)\\
&=&\fnorm{h'}\Biggl(-\int_a^x sp(s)ds+F(x)\Bigl(\int_a^x sp(s)ds+\int_x^b sp(s)ds\Bigr)\Biggr)\\
&=&\fnorm{h'}\Biggl(F(x)E[Z]-\int_a^x yp(y)dy\Biggr)\,,
\end{eqnarray*}

implying (a).\\
Now, we turn to the proof of (b). By Stein's equation (\ref{gensteineq}) we obtain for $x\in(a,b)$

\begin{equation}\label{eqb1}
g_h'(x)=\frac{1}{\eta(x)}\Bigl(\htilde(x)-\ga(x)g_h(x)\Bigr)\,,
\end{equation}

where we have again written $\htilde=h-\mu(h)$. Using Lemma \ref{liplemma} again, we obtain 

\begin{eqnarray}\label{eqb2}
g_h'(x)&=&\frac{1}{\eta(x)}\Biggl(\int_a^xF(s)h'(s)ds\Bigl(1+\frac{\ga(x)(1-F(x))}{\eta(x)p(x)}\Bigr)\nonumber\\
&&+\int_x^b(1-F(s))h'(s)ds\Bigl(-1+\frac{\ga(x)F(x)}{\eta(x)p(x)}\Bigr)\Biggr)\nonumber\\
&=&\int_a^xF(s)h'(s)ds\Bigl(\frac{\eta(x)p(x)+\ga(x)(1-F(x))}{\eta(x)^2p(x)}\Bigr)\nonumber\\
&&+\int_x^b(1-F(s))h'(s)ds\Bigl(\frac{-\eta(x)p(x)+\ga(x)F(x)}{\eta(x)^2p(x)}\Bigr)
\end{eqnarray}

Now consider the functions $H,G:\abquer\rightarrow\R$ with $H(x)=I(x)-\ga(x)F(x)=\eta(x)p(x)-\ga(x)F(x)$ and 
$G(x)=H(x)+\ga(x)=\eta(x)p(x)+\ga(x)(1-F(x))$. It was already observed in the proof of Proposition \ref{genprop3} that $H$ is positive on $(a,b)$. Similarly we prove the positivity of $G$ on $(a,b)$: For $x$ in $(a,b)$ we have, since $p$ is positive and $\ga$ is strictly decreasing 

\begin{eqnarray*}
G(x)&=&I(x)+\ga(x)(1-F(x))=-\int_x^b\ga(t)p(t)dt+\ga(x)(1-F(x))\\
&>&-\ga(x)(1-F(x))+\ga(x)(1-F(x))=0\,.
\end{eqnarray*}

By (\ref{eqb2}) we can thus bound

\begin{eqnarray}\label{eqb3}
|g_h'(x)|&\leq&\fnorm{h'}\Biggl(\int_a^xF(s)ds\frac{G(x)}{\eta(x)^2p(x)}\nonumber\\
&&+\int_x^b(1-F(s))ds\frac{H(x)}{\eta(x)^2p(x)}\Biggr)\,,
\end{eqnarray}

which reduces to the bound asserted in (b).
\end{proof}

\begin{proof}[Proof of Proposition \ref{propout3}]
The bound on $\abs{g_h(x)}$for $x\in(a,b)$ has already been proved in Proposition \ref{genprop3}. Let $x\in(-\infty,a)$. Then we have by the negativity of $q_l$ which follows from Proposition \ref{propout2}: 

\begin{eqnarray*}
\abs{g_h(x)}&=&\frac{\Bigl|\int_a^x\tilde{h}(t)q_l(t)dt\Bigr|}{\exp\bigl(F_l(x)\bigr)}
\leq\fnorm{\htilde}\frac{\Bigl|\int_a^xq_l(t)dt\Bigr|}{\exp\bigl(F_l(x)\bigr)}\\
&=&\fnorm{\htilde}\frac{\int_a^xq_l(t)dt}{\exp\bigl(F_l(x)\bigr)}=\fnorm{\htilde}\frac{Q_l(x)}{\exp\bigl(F_l(x)\bigr)}
\end{eqnarray*}

We want to show, that this is bounded from above by $\frac{\fnorm{\htilde}}{\ga(a)}$. To this end, we define the function 
$D(x):=\frac{\exp\bigl(F_l(x)\bigr)}{\ga(a)}-Q_l(x)$, $x\in(-\infty,a)$, and show that $D(x)>0$. 
By Condition \ref{condout2} and Proposition \ref{propout2} we have $\lim_{x\nearrow a}D(x)=0$ and furthermore 

\[D'(x)= \frac{\frac{\ga(x)}{\eta(x)}\exp\bigl(F_l(x)\bigr)}{\ga(a)}-q_l(x)
=q_l(x)\Bigl(\frac{\ga(x)}{\ga(a)}-1\Bigr)<0\,,\]

since $q_l(x)<0$ and $\ga(x)>\ga(a)$ by Condition \ref{condout1}. Thus, $D$ is strictly decreasing and hence $D(x)>0$ for each $x\in(-\infty,a)$. This proves the desired bound for $x\in(-\infty,a)$. Similarly one proves that 
$\abs{g_h(x)}\leq-\frac{\fnorm{\htilde}}{\ga(b)}$ for each $x\in(b,\infty)$.
\end{proof}

Next, we will state a lemma, which replaces Lemma \ref{genlemma1} outside the support $\abquer$. 

\begin{lemma}\label{outlemma1}
\begin{enumerate}[{\normalfont (a)}]
 \item For each $x\in(-\infty,a)$ we have $\int_x^a Q_l(t)dt=-xQ_l(x)+\int_a^x tq_l(t)dt$ and $I_l(x)<\ga(x)Q_l(x)$.
 \item For each $x\in(b,\infty)$ we have $\int_b^x Q_r(t)dt =xQ_r(x)-\int_b^x tq_r(t)dt$ and $I_r(x)<\ga(x)Q_r(x)$. 
\end{enumerate}
\end{lemma}

\begin{proof}
By Fubini's theorem we have

\begin{eqnarray*}
 \int_x^a Q_l(t)dt&=&\int_x^a\Biggl(\int_a^t q_l(s)ds\Biggr)dt=-\int_x^a\Biggl(\int_t^a q_l(s)ds\Biggr)dt\\
&=&-\int_x^a\Biggl(\int_x^s 1dt\Biggr)q_l(s)ds=-\int_x^a(s-x)q_l(s)ds\\
&=&-xQ_l(x)+\int_a^xsq_l(s)ds\,,
\end{eqnarray*}

which proves the first part of (a). The second claim of (a) follows from (a), the positivity of $-q_l$ on $(-\infty,a)$ and from the monotonicity of $\ga$:

\begin{eqnarray*}
I_l(x)&=&\int_a^x \ga(t)q_l(t)dt=\int_x^a\ga(t)\bigl(-q_l(t)\bigr)dt\\
&<&\ga(x)\int _x^a\bigl(-q_l(t)\bigr)dt=\ga(x)\int_a^x q_l(t)dt\\
&=&\ga(x) Q_l(x)
\end{eqnarray*}

 The proof of (b) is similar but easier, and is therefore omitted.
\end{proof}

The next lemma replaces Lemma \ref{liplemma} outside of the support of $\mu$. 

\begin{lemma}\label{outlemma2}
\begin{enumerate}[{\normalfont (a)}]
\item For each $x\in(-\infty,a)$ we have \\
$h(x)-\mu(h)=- \int_x^b\bigl(1-F(s)\bigr)h'(s)ds=-\int_x^a h'(s)ds - \int_a^b\bigl(1-F(s)\bigr)h'(s)ds$ and 
\begin{eqnarray*}
g_h(x)&=&\frac{1}{I_l(x)}\Biggl(-\int_x^a\bigl(Q_l(x)-Q_l(s)\bigr)h'(s)ds
-Q_l(x)\int_a^b\bigl(1-F(s)\bigr)h'(s)ds\Biggr)\\
&=&\frac{1}{I_l(x)}\Biggl(\int_x^a Q_l(s)h'(s)ds-Q_l(x)\int_x^b\bigl(1-F(s)\bigr)h'(s)ds\Biggr)\,.
\end{eqnarray*}
\item For each $x\in(b,\infty)$ we have $h(x)-\mu(h)=\int_a^x F(s)h'(s)ds$ and 
\begin{eqnarray*}
g_h(x)&=&\frac{1}{I_r(x)}\Biggl(Q_r(x)\int_a^b F(s)h'(s)ds+\int_b^x h'(s)\bigl(Q_r(x)-Q_r(s)\bigr)ds\Biggr)\\
&=&\frac{1}{I_r(x)}\Biggl(Q_r(x)\int_a^xF(s)h'(s)ds-\int_b^xQ_r(s)h'(s)ds\Biggr)\,.
\end{eqnarray*}
\end{enumerate}
\end{lemma}

\begin{proof}
We only prove (a) since the proof of (b) is very similar. The first claim follows from Lemma \ref{liplemma} (a) since $F(s)=0$ for 
$s<a$ and $F(s)=1$ for $x\geq b$. The second claim follows from the first one and from Fubini's theorem by

\begin{eqnarray*}
&&g_h(x)\\
&=&\frac{1}{I_l(x)}\int_a^x\htilde(y)q_l(y)dy=\frac{1}{I_l(x)}\int_a^x\Biggl(-\int_y^b\bigl(1-F(s)\bigr)h'(s)ds\Biggr)q_l(y)dy\\
&=&\frac{1}{I_l(x)}\int_x^a\Biggl(\int_y^b\bigl(1-F(s)\bigr)h'(s)ds\Biggr)q_l(y)dy\\
&=&\frac{1}{I_l(x)}\int_x^b\bigl(1-F(s)\bigr)h'(s)\Biggl(\int_x^{a\wedge s}q_l(y)dy\Biggr)ds\\
&=&\frac{1}{I_l(x)}\int_x^b\bigl(1-F(s)\bigr)h'(s)\bigl(Q_l(a\wedge s)-Q_l(x)\bigr)ds\\
&=&\frac{1}{I_l(x)}\int_x^a h'(s)\bigl(Q_l(s)-Q_l(x)\bigr)ds+\frac{1}{I_l(x)}\int_a^b\bigl(1-F(s)\bigr)h'(s)ds\bigl(Q_l(a)-Q_l(x)\bigr)\\
&=&\frac{1}{I_l(x)}\Biggl(-\int_x^a\bigl(Q_l(x)-Q_l(s)\bigr)h'(s)ds
-Q_l(x)\int_a^b\bigl(1-F(s)\bigr)h'(s)ds\Biggr)\,.
\end{eqnarray*}

This is the first representation for $g_h(x)$ in the assertion. The second one follows, since $1-F(s)=1$ for $s<a$ and hence

\begin{eqnarray*}
 -\int_x^a\bigl(Q_l(x)-Q_l(s)\bigr)h'(s)ds=\int_x^a Q_l(s)h'(s)ds
-Q_l(x)\int_x^a(1-F(s))h'(s)ds\,.
\end{eqnarray*}

\end{proof}

\begin{proof}[Proof of Proposition \ref{propout4}]
We only prove (a) and (c), since the proofs of (b) and (d) are similar. To prove (a), we observe that by Lemma \ref{outlemma2} we have 
\[g_h(x)=\frac{1}{I_l(x)}\Biggl(-\int_x^a\bigl(Q_l(x)-Q_l(s)\bigr)h'(s)ds
-Q_l(x)\int_a^b\bigl(1-F(s)\bigr)h'(s)ds\Biggr)\,.\]
Since $Q_l$ is decreasing ($Q_l'=q_l<0$) and positive on $(-\infty,a)$ this implies

\[\abs{g_h(x)}\leq\frac{\fnorm{h'}}{I_l(x)}\Biggl(\int_x^a\bigl(Q_l(x)-Q_l(s)\bigr)ds+Q_l(x)\int_a^b\bigl(1-F(s)\bigr)ds\Biggr)\,.\]

By Lemma \ref{genlemma1} and Lemma \ref{outlemma1} (a) the right hand side equals

\begin{eqnarray*}
&&\frac{\fnorm{h'}}{I_l(x)}\Biggl((a-x)Q_l(x)+xQ_l(x)-\int_a^x sq_l(s)ds+ Q_l(x)\bigl(E[Z]-a\bigr)\Biggr)\\
&=&\frac{\fnorm{h'}}{I_l(x)}\Biggl(Q_l(x)E[Z]-\int_a^x sq_l(s)ds\Biggr)\,,
\end{eqnarray*}

which is the claimed bound. Now, we turn to the proof of (c). By Stein's equation (\ref{gensteineq2}) and Lemma \ref{outlemma2} (a) we have for each $x\in(-\infty,a)$:

\begin{eqnarray*}
&&g_h'(x)=\frac{\htilde(x)}{\eta(x)}-\frac{\ga(x)}{\eta(x)}g_h(x)=-\frac{1}{\eta(x)}\int_x^b\bigl(1-F(s)\bigr)h'(s)ds\\
&&- \frac{\ga(x)}{\eta(x)I_l(x)}\Biggl(\int_x^a Q_l(s)h'(s)ds-Q_l(x)\int_x^b\bigl(1-F(s)\bigr)h'(s)ds\Biggr)\\
&=&\frac{-1}{\eta(x)I_l(x)}\Biggl(\ga(x)\int_x^aQ_l(s)h'(s)ds+\int_x^b\bigl(1-F(s)\bigr)h'(s)ds\Bigl(I_l(x)-Q_l(x)\ga(x)\Bigr)\Biggr)
\end{eqnarray*}

By Lemma \ref{genlemma1} (a) we have 

\[\int_x^b\bigl(1-F(s)\bigr)ds=a-x+\int_a^b\bigl(1-F(s)\bigr)ds=a-x+E[Z]-a=E[Z]-x\,.\]

Hence, by Lemma \ref{outlemma1} (a) this implies

\begin{eqnarray*}
&&\abs{g_h'(x)}\leq\frac{\fnorm{h'}}{-\eta(x)I_l(x)} \Biggl(\ga(x)\int_x^a Q_l(s)ds
+\int_x^b \bigl(1-F(s)\bigr)ds\Bigl(Q_l(x)\ga(x)-I_l(x)\Bigr)\Biggr)\\
&=&\frac{\fnorm{h'}}{-\eta(x)I_l(x)}\Biggl(\ga(x)\Bigl(-xQ_l(x)+\int_a^xtq_l(t)dt\Bigr)+\bigl(E[Z]-x\bigr) \Bigl(Q_l(x)\ga(x)-I_l(x)\Bigr)\Biggr)
\end{eqnarray*}

which was to be shown.
\end{proof}

\begin{proof}[Proof of Corollary \ref{outcor1}]
In this case, we have the relations 

\begin{eqnarray}\label{Igas}
I_l(x)&=&c\int_a^x\bigl(E[Z]-t\bigr)q_l(t)dt=c\Bigl(Q_l(x)E[Z]-\int_a^xt q_l(t)dt\Bigr)\\
I_r(x)&=&c\int_b^x\bigl(E[Z]-t\bigr)q_r(t)dt=c\Bigl(Q_r(x)E[Z]-\int_b^xt q_r(t)dt\Bigr)\,.\nonumber
\end{eqnarray}

These together with Proposition \ref{propout4} (a), (b) and Corollary \ref{gencor2} immediately imply (a). As to (b), by (\ref{Igas}) we have 

\begin{eqnarray*}
&&\bigl(E[Z]-x\bigr) \Bigl(Q_l(x)\ga(x)-I_l(x)\Bigr)\\
&=&\bigl(E[Z]-x\bigr)\Bigl(cE[Z]Q_l(x)-cxQ_l(x)-c\bigl(Q_l(x)E[Z]-\int_a^xt q_l(t)dt\bigr)\Bigr)\\
&=&c\bigl(E[Z]-x\bigr)\Bigl(-xQ_l(x)+\int_a^xt q_l(t)dt\Bigr)\\
&=&\ga(x)\Bigl(-xQ_l(x)+\int_a^xt q_l(t)dt\Bigr)\,.
\end{eqnarray*}

This and Proposition \ref{propout4} (c) imply claim (b). Assertion (c) may be proved similarly.
\end{proof}

\begin{proof}[Proof of Proposition \ref{genpropgz}]
For any $x\in(a,b)$ we have:
\begin{eqnarray*}
g_z(x)&=&\frac{1}{I(x)}\int_a^x\bigl(1_{(-\infty,z]}(t)-P(Z\leq z)\bigr)p(t)dt\\
&=&\frac{1}{I(x)}\Bigl(\int_a^{x\wedge z} p(t)dt-F(z)F(x)\Bigr)\\
&=&\frac{F(x\wedge z)-F(z)F(x)}{I(x)}\,,
\end{eqnarray*}

proving the desired representation of $g_z$ inside the interval $(a,b)$.
Now, for $x\in(a,b)$ let $M(x):=\frac{F(x)}{I(x)}$ and $N(x):=\frac{1-F(x)}{I(x)}$. Then we have 

\begin{eqnarray*}
M'(x)&=&\frac{p(x)I(x)-p(x)\ga(x)F(x)}{I(x)^2}=\frac{p(x)}{I(x)^2}\Bigl(I(x)-\ga(x)F(x)\Bigr)\\
&=&\frac{p(x)}{I(x)^2}H(x)>0
\end{eqnarray*}

and

\begin{eqnarray*}
N'(x)&=&\frac{-p(x)I(x)-p(x)\ga(x)(1-F(x))}{I(x)^2}=\frac{-p(x)}{I(x)^2}\Bigl(I(x)+\bigl(1-F(x)\bigr)\ga(x)\Bigr)\\
&=&\frac{-p(x)}{I(x)^2}G(x)<0\,,
\end{eqnarray*}

since $G(x)=(I(x)+(1-F(x))\ga(x))$ is positive by Proposition \ref{genprop4}. Thus, $M$ is strictly increasing and $N$ is strictly decreasing on $(a,b)$. Since 
$g_z(x)=(1-F(z))M(x)$ for $x\in(a,z]$ and $g_z(x)=F(z)N(x)$ for $x\in(z,b)$, this implies, that 
$\sup_{x\in(a,b)}g_z(x)=g_z(z)=\frac{F(z)(1-F(z)}{I(z)}$. It also implies the claimed representation of $g_z'(x)$ for 
$x\in(a,b)\setminus\{z\}$. 
Furthermore, by de l'H\^{o}pital's rule, we have 

\begin{eqnarray*}
\lim_{x\searrow a}g_z(x)&=&(1-F(z))\lim_{x\searrow a}M(x)=(1-F(z))\lim_{x\searrow a}\frac{p(x)}{p(x)\ga(x)}\\
&=&\lim_{x\searrow a}\frac{(1-F(z))}{\ga(x)}=\frac{(1-F(z))}{\ga(a)}
\end{eqnarray*}

and

\begin{eqnarray*}
\lim_{x\nearrow b}g_z(x)&=&F(z)\lim_{x\nearrow b}N(x)=F(z)\lim_{x\nearrow b}\frac{-p(x)}{p(x)\ga(x)}\\
&=&\lim_{x\nearrow b}\frac{-F(z)}{\ga(x)}=\frac{-F(z)}{\ga(b)}\,.
\end{eqnarray*}

Note, that these limits could also be derived from Proposition \ref{propout1}.\\
Next, consider $x\in(-\infty,a)$. For such an $x$ we have

\[g_z(x)=\frac{1}{I_l(x)}\int_a^x\bigl(1_{(-\infty,z]}(t)-P(Z\leq z)\bigr)q_l(t)dt=\frac{(1-F(z))Q_l(x)}{I_l(x)}\,.\]

Moreover we have

\begin{eqnarray*}
\frac{d}{dx}\frac{Q_l(x)}{I_l(x)}&=&\frac{q_l(x)I_l(x)-\ga(x)q_l(x)Q_l(x)}{I_l(x)^2}\\
&=&\frac{-q_l(x)}{I_l(x)^2}\Bigl(\ga(x)Q_l(x)-I_l(x)\Bigr)>0
\end{eqnarray*}

by Lemma \ref{outlemma1}. Thus, $g_z$ is increasing on $(-\infty,a)$ and hence, again by de l'H\^{o}pital's rule, 

\begin{eqnarray*}
\sup_{x\in(-\infty,a)}\abs{g_z(x)}&=&\sup_{x\in(-\infty,a)}g_z(x)=(1-F(z))\lim_{x\nearrow a}\frac{Q_l(x)}{I_l(x)}\\
&=&(1-F(z))\lim_{x\nearrow a}\frac{q_l(x)}{q_l(x)\ga(x)}=(1-F(z))\lim_{x\nearrow a}\frac{1}{\ga(x)}\\
&=&\frac{1-F(z)}{\ga(a)}\,.
\end{eqnarray*}

Since $g_z'(x)=(1-F(z))\frac{d}{dx}\frac{Q_l(x)}{I_l(x)}$ we have also derived the desired formula for $g_z'(x)$ for 
$x\in(-\infty,a)$. The calculations for $x\in(b,\infty)$ are completely analogous and therefore omitted. From our computations we can already infer, that $\fnorm{g_z}=\frac{F(z)(1-F(z)}{I(z)}$. So, it remains to show that this quantity is bounded in $z\in(a,b)$. Since it is a continuous function of $z$, we only have to show that it has finite limits on the edge of the interval $(a,b)$. But, of course, 

\[\lim_{z\searrow a}\frac{F(z)(1-F(z)}{I(z)}=\lim_{z\searrow a}\frac{F(z)}{I(z)}=\lim_{z\searrow a}M(z)=\frac{1}{\ga(a)}\]

and, similarly, $\lim_{z\nearrow b}\frac{F(z)(1-F(z)}{I(z)}=\lim_{z\nearrow b}N(z)=-\frac{1}{\ga(b)}$. This concludes the proof.
\end{proof}

\subsection{Proofs from Section \ref{mainresultsbeta}}

\begin{proof}[Proof of Proposition \ref{gzk}]
Let us focus on the case $z\in(-1,1)$, the cases $z<-1$ and $z>1$ being similar and even easier. Of course, $g_z$ is continuously differentiable on each of the intervals $(-\infty,-1)$, $(-1,z)$, $(z,1)$ and $(1,\infty)$ and from Proposition \ref{propout1} we know that $g_z$ is continuous on $\R$ since $z\not=\pm1$. By a result from calculus, to show that, for example, $g_z$ is $C^1$ on $[-1,z]$, it is sufficient to prove that $\lim_{x\searrow-1}g_z'(x)$ exists in $\R$. Tot this end we recall from Remark \ref{outrem2} that for $-1<x<z$ we have 

\[g_z'(x)=\bigl(1-F(z)\bigr)\frac{p(x)H(x)}{I(x)^2}= \bigl(1-F(z)\bigr)\frac{H(x)}{(1-x^2)p(x)}\,.\]

Using de l'H\^{o}pital's rule \ref{Hospital} and Lemma \ref{betalemma1} (a) below, we obtain
\begin{eqnarray*}
\lim_{x\searrow-1}g_z'(x)&=&\bigl(1-F(z)\bigr) \lim_{x\searrow-1}\frac{H(x)}{(1-x^2)p(x)}\\
&=&\bigl(1-F(z)\bigr) \lim_{x\searrow-1}\frac{(\al+\be+2)F(x)}{(1-x^2)p(x)[\be-\al-(\al+\be+4)x]}\\
&=&\frac{1-F(z)}{2\be+4}\lim_{x\searrow-1}\frac{(\al+\be+2)F(x)}{(1-x^2)p(x)}\\
&=&\frac{1-F(z)}{2\be+4}\lim_{x\searrow-1}\frac{(\al+\be+2)p(x)}{p(x)[\be-\al-(\al+\be+2)x]}\\
&=&\frac{\bigl(1-F(z)\bigr)\bigl(\al+\be+2\bigr)}{(2\be+4)(2\be+2)}\,.
\end{eqnarray*}

Hence, $g_z$ is continuously differentiable on $[-1,z]$. That $g_z$ is continuously differentiable on the intervals $(-\infty,1]$, 
$[z,1]$ and $[1,\infty)$ can be proved similarly by using the representations for $g_z'(x)$ from Proposition \ref{genpropgz}. It also turns out, that the right and left hand derivatives of $g_z$ at $\pm1$ coincide, so that it is in fact differentiable at these points.\\
So, we conclude, that $g_z$ is continuous and piecewise continuously differentiable on $\R$. Next, we show that $g_z$ vanishes at infinity. This even applies to any bounded test function $h$, since e.g. for $x>1$ we have

\begin{eqnarray*}
\abs{g_h(x)}&=&\Bigl|\exp(-F_r(x))\int_1^x\htilde(t)q_r(t)dt\Bigr|\\
&\leq&\fnorm{\htilde}\frac{\int_1^x q_r(t)dt}{(x-1)^{\al+1}(1+x)^{\be+1}} 
\end{eqnarray*}

and by de l'H\^{o}pital's rule

\begin{eqnarray*}
&&\lim_{x\to\infty} \frac{\int_1^x q_r(t)dt}{(x-1)^{\al+1}(1+x)^{\be+1}}
=\lim_{x\to\infty} \frac{\int_1^x(t-1)^{\al}(1+t)^{\be}dt}{(x-1)^{\al+1}(1+x)^{\be+1}}\\
&=&\lim_{x\to\infty}\frac{(x-1)^{\al}(1+x)^{\be}}{(x-1)^{\al}(1+x)^{\be}\bigl((\al+1)(1+x)+(\be+1)(x-1)\bigr)}\\
&=&\lim_{x\to\infty}\frac{1}{(\al+1)(1+x)+(\be+1)(x-1)}=0\,.
\end{eqnarray*}

Hence, $ \lim_{x\to\infty}g_h(x)=0$. Similarly, one can prove that $ \lim_{x\to-\infty}g_h(x)=0$. \\
It remains to show that 

\[\int_{-1}^{1}\abs{g_z'(x)}\rho_{\albe}(x)dx=\int_{-1}^1\abs{g_z'(x)}(1-x^2)p(x)dx<\infty\,.\]

To this end, it suffices to see that the function $ \abs{g_z'(x)}(1-x^2)p(x)$ is bounded on $(-1,z)$ and on $(z,1)$. 
Since it is continuous on $(-1,z]$ and on $(z,1)$ (where $g_z'(z):=\lim_{x\nearrow z}g_z'(x)$ for definiteness), this claim will follow 
if we have proved that $\lim_{x\to\pm1} g_z'(x)(1-x^2)p(x)=0$. For $x\in(-1,z)$ we have 

\[g_z'(x)(1-x^2)p(x)=\bigl(1-F(z)\bigr)\frac{H(x)}{1-x^2}\quad\text{and}\]

\[\lim_{x\searrow-1}\frac{H(x)}{1-x^2}=\lim_{x\searrow-1}\frac{(\al+\be+2)F(x)}{-2x}=0\,,\]

proving the claim for $-1$. Since it may be proved analogously for $+1$ the proof is complete.
\end{proof}

The following lemma will be useful for the proof of Proposition \ref{propbeta2}.

\begin{lemma}\label{betalemma1}
The functions $p$, $q_l$, $q_r$ and $\eta$ satisfy the following equations for each integer $k\geq1$:
\begin{enumerate}[{\normalfont (a)}]
\item $\frac{d}{dx}\bigl(\eta(x)^kp(x)\bigr)=\eta(x)^{k-1}p(x)\bigl[(k-1)\eta'(x)+\ga(x)\bigr]$ for each $x\in(-1,1)$
\item $\frac{d}{dx}\bigl(\eta(x)^kq_l(x)\bigr)=\eta(x)^{k-1}q_l(x)\bigl[(k-1)\eta'(x)+\ga(x)\bigr]$ for each $x\in(-\infty,-1)$
\item $\frac{d}{dx}\bigl(\eta(x)^kq_r(x)\bigr)=\eta(x)^{k-1}q_r(x)\bigl[(k-1)\eta'(x)+\ga(x)\bigr]$ for each $x\in(1,\infty)$
\end{enumerate}
\end{lemma}

\begin{proof}
First we prove (a). By (\ref{etadgl}) we have, multiplying by $p(x)$,

\[\eta(x)p'(x)=p(x)\bigl(\ga(x)-\eta'(x)\bigr)\,.\]

Hence, by the product rule we obtain

\begin{eqnarray*}
\frac{d}{dx}\bigl(\eta(x)^kp(x)\bigr)&=&k\eta(x)^{k-1}\eta'(x)p(x)+\eta(x)^kp'(x)\\
&=&\eta(x)^{k-1}\bigl[k\eta'(x)p(x)+\eta(x)p'(x)\bigr]\\
&=&\eta(x)^{k-1}p(x)\bigl[k\eta'(x)+\ga(x)-\eta'(x)\bigr]\\
&=&\eta(x)^{k-1}p(x)\bigl[(k-1)\eta'(x)+\ga(x)\bigr]\,,
\end{eqnarray*}

proving (a). As to (b), we observe that by the definition of $q_l=\frac{\exp\circ F_l}{\eta}$ we have on the one hand

\begin{eqnarray*}
\frac{d}{dx}\bigl(\eta(x)q_l(x)\bigr)&=&\frac{d}{dx}\exp\bigl(F_l(x)\bigr)=\exp\bigl(F_l(x)\bigr)\frac{\ga(x)}{\eta(x)}\\
&=&\ga(x)q_l(x)
\end{eqnarray*}

and on the other hand, by the product rule,

\[\frac{d}{dx}\bigl(\eta(x)q_l(x)\bigr)=\eta'(x)q_l(x)+q_l'(x)\eta(x)\,.\]

These two equations yield

\[\eta(x)q_l'(x)=\ga(x)q_l(x)-\eta'(x)q_l(x)=q_l(x)\bigl(\ga(x)-\eta'(x)\bigr)\,.\]

Now the proof follows the lines of proof for $p$ as above. Similarly one may prove (c).
\end{proof}

\begin{proof}[Proof of Proposition \ref{propbeta2}]
Assertion (a) immediately follows from Corollary \ref{outcor1} (a) since in this case $c=\al+\be+2$. Now, we turn to the proof of (b). First, consider $x\in(-1,1)$. By Corollary \ref{gencor2} (b) we have for $x\in(-1,1)$:

\[\abs{g_h'(x)}\leq\frac{2\fnorm{h'}}{\al+\be+2}\frac{H(x)G(x)}{I(x)\eta(x)}\,.\]

For $x\in(-1,1)$ let

\[S(x):=\frac{H(x)G(x)}{I(x)\eta(x)}=\frac{H(x)G(x)}{\eta(x)^2p(x)}\,.\]

Then $S$ is continuous on $(-1,1)$ and hence, to show that $S$ is bounded on $(-1,1)$, it suffices to prove that $S$ has finite limits at $\pm1$. Since $\lim_{x\searrow-1}G(x)=\ga(-1)=2\be+2$ we obtain, using Lemma \ref{betalemma1} and de l'H\^{o}pital's rule:

\begin{eqnarray*}
&&\lim_{x\searrow-1}S(x)=\lim_{x\searrow-1}G(x)\lim_{x\searrow-1}\frac{H(x)}{\eta(x)^2p(x)}\\
&=&\ga(-1)\lim_{x\searrow-1}\frac{H'(x)}{\frac{d}{dx}\bigl(\eta(x)^2p(x)\bigr)}
=\ga(-1)\lim_{x\searrow-1}\frac{(\al+\be+2)F(x)}{\eta(x)p(x)\bigl[\eta'(x)+\ga(x)\bigr]}\\
&=&\frac{\ga(-1)(\al+\be+2)}{\lim_{x\searrow-1}\bigl[\eta'(x)+\ga(x)\bigr]}\lim_{x\searrow-1}\frac{F(x)}{\eta(x)p(x)}\\
&=&\frac{\ga(-1)(\al+\be+2)}{\lim_{x\searrow-1}\bigl[-2x-(\al+\be+2)x+\be-\al\bigr]}\lim_{x\searrow-1}\frac{p(x)}{\ga(x)p(x)}\\
&=&\frac{\ga(-1)(\al+\be+2)}{2\be+4}\lim_{x\searrow-1}\frac{1}{\ga(x)}
=\frac{\ga(-1)(\al+\be+2)}{2\be+4}\frac{1}{\ga(-1)}\\
&=&\frac{\al+\be+2}{2\be+4}<\infty
\end{eqnarray*}

Here we have used, that $H'(x)=-\ga'(x)F(x)=(\al+\be+2)F(x)$. Similarly one shows that

\[\lim_{x\nearrow1}S(x)=\frac{\al+\be+2}{2\al+4}\,.\]

Thus, we have shown that $\sup_{x\in(-1,1)}S(x)<\infty$. Now, we consider $x\in(-\infty,-1)$. From Corollary \ref{outcor1} (b) we have

\[\abs{g_h'(x)}\leq2\fnorm{h'}\frac{\ga(x)\Bigl(-xQ_l(x)+\int_a^xtq_l(t)dt\Bigr)}{-\eta(x)I_l(x)}\,.\]

For $x\in(-\infty,-1)$ we consider the function 

\[S_l(x):=\frac{\ga(x)\Bigl(-xQ_l(x)+\int_a^xtq_l(t)dt\Bigr)}{-\eta(x)I_l(x)}
=\frac{\ga(x)\int_a^xQ(t)dt}{\eta(x)^2q_l(x)}\,.\]

Clearly, $S_l$ is a continuous function on $(-\infty,-1)$. To show that it is bounded, it thus suffices to prove that 
$\lim_{x\nearrow-1}S_l(x)<\infty$ and $\lim_{x\to-\infty}S_l(x)<\infty$.
Using Lemma \ref{betalemma1} and de l'H\^{o}pital's rule, we obtain

\begin{eqnarray*}
&&\lim_{x\nearrow-1}S_l(x)=\lim_{x\nearrow-1}\ga(x)\lim_{x\nearrow-1}\frac{\int_a^xQ(t)dt}{\eta(x)^2q_l(x)}\\
&=&\ga(-1)\lim_{x\nearrow-1}\frac{Q_l(x)}{\eta(x)q_l(x)\bigl[\eta'(x)+\ga(x)\bigr]}\\
&=&\frac{\ga(-1)}{\lim_{x\nearrow-1}\bigl[\eta'(x)+\ga(x)\bigr]}\lim_{x\nearrow-1}\frac{Q_l(x)}{\eta(x)q_l(x)}\\
&=&\frac{\ga(-1)}{2\be+4}\lim_{x\nearrow-1}\frac{q_l(x)}{\ga(x)q_l(x)}
=\frac{\ga(-1)}{2\be+4}\lim_{x\nearrow-1}\frac{1}{\ga(x)}\\
&=&\frac{1}{2\be+4}<\infty\,.
\end{eqnarray*}

Next, we will show that $\lim_{x\to-\infty}S_l(x)=0$. We actually have

\[\lim_{x\to-\infty}\frac{\ga(x)}{\eta(x)}=\lim_{x\to-\infty}\frac{-(\al+\be+2)x+\be-\al}{1-x^2}=0\]

and by de l'H\^{o}pital's rule, also

\[\lim_{x\to-\infty}\frac{\int_a^xQ(t)dt}{\eta(x)q_l(x)}=\lim_{x\to-\infty}\frac{q_l(x)}{q_l(x)\ga(x)}
=\lim_{x\to-\infty}\frac{1}{\ga(x)}=0\,.\]

Here we have used that $\eta(x)q_l(x)=-(1-x)^{\al+1}(-1-x)^{\be+1}\rightarrow-\infty$ as $x\to-\infty$. 
Hence, $\lim_{x\to-\infty}S_l(x)=0$ and $\sup_{x\in(-\infty,-1)}S_l(x)<\infty$. Since we can show in a similar manner that 
$\sup_{x\in(1,\infty)}S_r(x)<\infty$, where $S_r$ is defined in the obvious way, the proof is complete.
\end{proof}

\begin{proof}[Proof of Lemma \ref{betalemma2}]
If $E[u(Z)]=0$, then the usual Stein solution is bounded on $(-1,1)$ by Proposition \ref{propbeta1}.
For the converse, let us assume that $E[u(Z)]\not=0$. As was already noted in Section \ref{mainresultsgen}, the solutions of the homogeneous equation corresponding to (\ref{steineqnc}) are exactly the multiples of $\frac{1}{\eta p}$.
Thus, every solution $f$ of (\ref{steineqnc}) has the form  

\[f(x)=\frac{\int_{-1}^x u(t)p(t)dt+c}{\eta(x)p(x)}\,,\quad x\in(-1,1)\,,\]

for some $c\in\R$. If $c=0$, then 

\[\lim_{x\nearrow1}f(x)=\lim_{x\nearrow1}\frac{\int_{-1}^x u(t)p(t)dt}{\eta(x)p(x)}=\pm\infty\,,\]

since $E[u(Z)]\not=0$. Hence, $f$ is unbounded near $1$. If $c\not=0$ we have 

\[\lim_{x\searrow-1}f(x)=\lim_{x\searrow-1}\frac{c}{\eta(x)p(x)}=\pm\infty\,.\]

So, $f$ is unbounded near $-1$. Hence, in any case $f$ is unbounded on $(-1,1)$.

\end{proof}

\normalem
\bibliography{steinbeta}{}
\bibliographystyle{alpha}

\end{document}